\documentclass{article}

\usepackage{graphicx,url}
\usepackage{color}
\newcount\Comments  
\newcommand{\kibitz}[2]{\ifnum\Comments=0\textcolor{#1}{#2}\fi}

\usepackage{amsmath}
\usepackage{amssymb,amsthm}
\usepackage[T1]{fontenc}
\usepackage{latexsym}
\usepackage{psfrag}
\usepackage{hyperref}
\usepackage{multirow}
\usepackage{tikz}
\usepackage{mathdots}
\usepackage{yhmath}
\usepackage{cancel}
\usepackage{color}
\usepackage{siunitx}
\usepackage{array,enumerate}
\usepackage{multirow}
\usepackage{amssymb}
\usepackage{gensymb}
\usepackage{tabularx}
\usepackage{booktabs}
\usepackage{dsfont,bbm,scalerel}
\usetikzlibrary{fadings}
\usetikzlibrary{patterns}
\usetikzlibrary{shadows.blur}
\usetikzlibrary{shapes}
\usepackage{pgf}
\usetikzlibrary{arrows,automata}

\DeclareMathAlphabet{\mathbx}{U}{BOONDOX-ds}{m}{n}

\newtheorem{theorem}{Theorem}
\newtheorem{prop}{Proposition}
\newtheorem{definition}{Definition}
\newtheorem{remark}{Remark}
\newtheorem{example}{Example}
\newtheorem{lemma}{Lemma}

\newcommand\reallywidehat[1]{\arraycolsep=0pt\relax%
\begin{array}{c}
\stretchto{
  \scaleto{
    \scalerel*[\widthof{\ensuremath{#1}}]{\kern-.5pt\bigwedge\kern-.5pt}
    {\rule[-\textheight/2]{1ex}{\textheight}} 
  }{\textheight} %
}{0.5ex}\\           
#1\\                 
\rule{-1ex}{0ex}
\end{array}
}

\newcommand{\R}{\mathbb R}

\newcommand{\N}{{\mathbb N}}

\newcommand{\PP}{\mathcal{P}}

\newcommand{\MM}{\mathcal{M}}

\newcommand{\eps}{\varepsilon}

\newcommand{\setdef}[2]{\left\{#1 \, : \; #2\right\}}

\newcommand{\real}{\mathbb{R}}

\newcommand{\be}{\begin{equation}}
\newcommand{\ee}{\end{equation}}
\newcommand{\ba}{\begin{array}}
\newcommand{\ea}{\end{array}}

\begin{document}

\title{Generalized solutions to opinion dynamics models with discontinuities}

\author{Francesca Ceragioli, Paolo Frasca, Benedetto Piccoli and Francesco Rossi
}                     
%
%
\maketitle
%


\begin{abstract}
Social dynamics models may present discontinuities in the right-hand side of the dynamics for multiple reasons, including topology changes and quantization.
Several concepts of generalized solutions for discontinuous equations are available in the literature and are useful to analyze these models.
In this chapter, we study Caratheodory and Krasovsky generalized solutions for discontinuous models of opinion dynamics with state dependent interactions. We consider two definitions of ``bounded confidence'' interactions, which we respectively call metric and topological: in the former, individuals interact if their opinions are closer than a threshold; in the latter, individuals interact with a fixed number of nearest neighbors. We compare the dynamics produced by the two kinds of interactions in terms of existence, uniqueness and asymptotic behavior of different types of solutions.
\end{abstract}


\section{Introduction and summary of results}
In the last decades, researchers from many different fields explored the behavior of large systems of active particles or agents.
The latter entities, also called self-propelled, intelligent or greedy, are endowed with the capability of decision making and, usually,
of altering the energy or other (otherwise conserved) quantities of the system. 
Examples include dynamics of opinions in social networks, animal groups,
networked robots, pedestrian dynamics and language evolution.
Their dynamics is written as an Ordinary Differential Equation (ODE in the following) in large dimension. In order to cope with this large dimension, various mean-field, kinetic
and hydrodynamic limit descriptions were studied in the literature, see
\cite{MBG16,ABFHKPPS19,Bertozzi,CCH,CFRT,CPT,DDM14,PR18} and references therein.
The interaction among agents may be restricted to specific regions due to the physical aspects of the modeled phenomenon, giving rise to discontinuities.
Even more, due to modeling choices these discontinuities may
appear naturally at multiple scales, see for instance
\cite{ABGR20,CPT14}. 

One of the main phenomena of active particles is \emph{self-organization}
of the whole system, stemming from simple interaction rules at the particle
level.  Such interaction rules are often motivated by relationships
among agents; thus, corresponding evolutions are referred to as \emph{social dynamics}
\cite{Piccoli:2017:chapter,PROSKURNIKOV201765,PROSKURNIKOV2018166}.
The most common self-organized configurations are: consensus  \cite{OFM07}, i.e.\ all agents reaching a common state; alignment, i.e.\ agents reaching consensus on a subset of the state variables (e.g.\ speed) \cite{CFPT15}; and clustering, i.e.\ agents grouping in a small number
of well-separated states \cite{JABIN20144165,MT14}.

The description of social dynamics may require continuous or discrete quantities, and the corresponding ODE may have either continuous or discontinuous vector fields. As a matter of fact, there are multiple situations where discrete variables and discontinuities arise. A partial list includes:
\begin{itemize}
\item[--] the presence of threshold effects caused by physical, communication, or psychological barriers (\cite{HK});
\item[--] the presence of quantities taking values in discrete sets, when a finite number of choices is given (e.g.\ whether and which product to buy) or when communication takes place by means of a finite set of symbols
(\cite{ceragioli2018consensus,Ch-Mo_Sr-CDC16,Martins2007CODA});
\item[--] the presence of a pattern of allowed/forbidden interactions, such as can be encoded in a graph
(\cite{Piccoli:2017:chapter,Ballerini:2008:evidence}).
\end{itemize}
The latter case includes all situations where physical or cognitive constraints limit interactions to agents that are close to each other, either spatially or behaviorally. 
When  models are defined in discrete time, discontinuities of the right-hand side pose little mathematical difficulties and are easily managed or simply ignored. Instead, {\em discontinuities give rise to technical difficulties in continuous time}, as the study of ODEs is deeply based on the notions of continuity and differentiability.
We notice that the problem of dealing with discontinuities is not limited to the case of ODE models, but it also occurs at other scales for learning dynamics in crowds combined
with loss of symmetry features, see
\cite{ABGR20}.

Even though discontinuities of some models can be avoided by defining suitable smoothed counterparts, which feature continuous approximations of discontinuous functions, the connections between continuous and discontinuous variants are not trivial~\cite{frasca}.
Most importantly, discontinuities cannot always be avoided. This is the case when the agents are allowed a finite number of choices/actions/interactions.
Another, classical, example of unavoidable discontinuity in ODEs comes from control theory. Whereas stabilizability of a system with inputs by means of a continuous feedback law implies asymptotic controllability by means of an open-loop control, the converse does not hold if only continuous feedback laws are considered.  But if  discontinuous feedback laws are allowed, then asymptotic controllability does imply stabilizability (\cite{AnconaFabio1999PVFa,CLSS}).  
This example shows not only that discontinuities cannot be avoided, but  they may be helpful. 
This result and, more generally, all results on discontinuous ODEs depend on the notion of generalized solutions adopted:  it is interesting to  analyse and compare different notions in order to deduce common features and differences.  This is one of the objectives of the present paper:  we will make use of the main concepts of solutions that have been defined in mathematical analysis and in control theory, and in particular we will discuss {\em classical, Caratheodory, Filippov and Krasovsky solutions}. We recall the precise definitions of these solutions in Section~\ref{s-sols} below.

We now describe more precisely the two opinion dynamics models that we analyze in the present chapter. 
The fact that an individual influences those he communicates with, can be taken as a principle when describing evolution of opinions. The most  basic model which describes in a mathematical framework this principle is usually referred to as DeGroot's model (despite having earlier origin in French~\cite{french1956formal}). Its main feature is that in a group of  individuals that communicate among them, consensus  is achieved.  On the other hand, everyone's experience suggests that consensus is not always achieved among individuals. For this reason,  many researchers have proposed more complex models, aiming to describe agreement and disagreement at the same time: see \cite{PROSKURNIKOV201765,PROSKURNIKOV2018166} for a comprehensive discussion. Crucially, several of these more complex models feature discontinuities of the right-hand side.

Here we consider a general model, which incorporates the well known Hegselmann-Krause model \cite{HK}. The basic idea of Hegselmann and Krause is that trust towards others has some limitations. In their work, they assume that an individual is influenced by others only if  opinions are not too far from one another.
Here, we describe the fact that one's confidence towards others  is limited, by describing in two different ways the set of neighbors of an individual. 
In the first setting, interactions among individuals follow  Hegselmann and Krause's rule:  one's neighbors are individuals whose  opinions do not differ too much. We call this  kind of interactions {\em metric interactions}.
In the second setting, we assume that an individual follows only a fixed number of neighbors, the ones whose opinions are the nearest to his own. 
We call this kind of interactions
{\em topological interactions}. 
Topological interactions can be motivated by the notion of Dunbar number \cite{Ballerini:2008:evidence,DUNBAR1992469} that indicates a cognitive limit in the number of significant relationships among individuals. This concept is particularly meaningful in the contemporary world,  where potential contacts and available information seem to be unlimited.

\subsection{Mathematical models and main results}
In the mathematical description of the bounded confidence models, we start by considering a set   $V=\{1,\ldots,N\}$ of $N$ agents (also called individuals)  with states $x_i\in \R^n$ (e.g.\ position, opinion, speed). Each agent $i\in V$ interacts with other agents belonging to a subset of neighbors $N_i(x)\subseteq V$. The subset of neighbors  $N_i(x)$ depends on the state and induces a graph $G(x)$ of interactions among the agents: $V$ is the set of nodes and $(i,j)$ 
is an edge if $j\in N_i(x)$. We denote the set of edges by $E(x)$. The dynamics can be written in the following form:

\begin{equation}\label{eq:general}
\dot{x}_i=\sum_{j\in N_i(x)} a(\| x_j-x_i \|) (x_j-x_i).
\end{equation}
The function $a:[0,+\infty[\to [0,+\infty[$ satisfies the following hypotheses:
\begin{itemize}
    \item $a$ is Lipschitz continuous;
    \item $a(r)>0$ for $r>0$; 
    \item $a$ is not decreasing.
\end{itemize}
The function $a$ represents the strength of interactions
among agents. A more general model could be written with interaction functions $a_{ij}$ that depend on the pair of neighbors. Most results stated in this article remain valid in this more general setting, provided that interactions are symmetric ($a_{ij}=a_{ji}$). 
Depending on how neighbors $N_i(x)$ are chosen, one obtains different bounded confidence models. From now on, we will use the notations $N^m_i,N^t_i$ for the set of neighbors for the metric and topological versions, that we make explicit below.

In the \emph{metric bounded confidence model} 
  agent $i$'s neighbors are those whose state is not too far from  his own, namely
$$ N_i^{m}(x)=\{j\in V: \|x_j-x_i\|<1\}.
$$
This choice implies that interactions
between agent $i$ and $j$ are symmetric, i.e.\ 
agent $i$ is influenced by agent $j$ if and only if agent $j$ is influenced by agent $i$. We then write the metric bounded confidence dynamics as follows:  
\begin{equation}\label{eq:mbc}
\dot{x}_i=\sum_{j\in N_i^m(x)} a(\| x_j-x_i \|) (x_j-x_i).
\end{equation}
As already mentioned, the first and best known version of the metric bounded confidence model is Hegselmann-Krause's \cite{blondelHK,HK,YY-DVD-XH:14}, which 
corresponds to  $a\equiv 1$ and was originally written in discrete time with states $x_i\in \R$. Its continuous-time counterpart was first studied in~\cite{blondel2}.

The \emph{topological bounded confidence model} 
is obtained when agent $i$  interacts only  with a fixed number $\kappa$ of neighbors, where $1\leq \kappa\leq N-1$. More precisely, for every agent $i\in V$, her neighborhood $N_i^t(x) $ is defined in the following way. The elements of $V\setminus\{i\}$ are ordered by increasing values of $\| x_j-x_i\|$; then, the first $\kappa$ elements of the list (i.e.\ those with smallest distance from $i$) form the set $N_i^t(x)$ of current neighbors of $i$.
Should a tie between two or more agents arise, priority is given to agents with lower index. We then write \begin{equation}\label{eq:tbc}
\dot{x}_i=\sum_{j\in N_i^t(x)} a(\| x_j-x_i \|) (x_j-x_i).
\end{equation}
For topological interactions, agent $i$ could be influenced  by agent $j$ without agent $j$ being influenced by agent $i$, namely interactions are not symmetric. This fact is a major difference between the metric and topological bounded confidence models.
This model was first pointed out in \cite{Piccoli:2017:chapter}, while several other models of opinion dynamics and collective motion have considered topological interactions in different forms: see~\cite{cristiani2011effects,rossi2020opinion} and references therein.


For both models, we have the following crucial observation: the right hand side of \eqref{eq:mbc}-\eqref{eq:tbc} is a discontinuous function. For this reason, one needs to carefully select a concept of solution to such discontinuous ODE.
In our opinion, this aspect has been overlooked in the extensive literature about bounded-confidence models, with some  exceptions such as \cite{blondel1,blondel2,frasca}.
Here, we will consider mainly Caratheodory and Krasovsky solutions. Definitions and a brief discussion on different notions of solutions can be found in Section~\ref{s-sols}. 
The first  result about solutions of \eqref{eq:general} will be the following.


\begin{theorem}[Existence and uniqueness] \label{t-exun}
Consider the bounded confidence models, either in the metric case \eqref{eq:mbc} or topological case \eqref{eq:tbc}.
Then, there exists a solution (global in time) for every initial condition
in the Krasovsky sense.
Uniqueness of solutions does not hold in general, but  holds for almost every initial datum.
Moreover, the same result holds for Caratheodory solution, both in the metric case and in the topological case for $\kappa=1$.
\end{theorem}
The full proof of the result for \eqref{eq:mbc} was given in \cite{piccoli_rossi_2020}, which extended partial results from~\cite{blondel2,frasca,ceragioli2018discontinuities}. Here we prove the corresponding claims for \eqref{eq:tbc} in Section~\ref{sect:solutions}.

\smallskip
After solving the questions about existence and uniqueness, we focus on some properties of such solutions that have been explored in the rich literature about social dynamics models.  We want to recall some of them. In the next definitions  $x(t)= (x_1(t),...,x_N(t))$ will denote a solution of an unspecified type. 

\begin{itemize}
\item[{\bf P1)}]  {\bf Average preservation.}  $x_{ave}(t)=\frac{1}{N}\sum_i x_i (t) $ is 
invariant along trajectories.

\item[{\bf P2)}]
{\bf Contractivity of the support.} For all $T^2\geq T^1\geq 0$, it holds
$$
\overline{co}\left(\left\{x_1(T^1),x_2(T^1),\ldots,x_N(T^1)\right\}\right)\supseteq \overline{co}\left(\left\{x_1(T^2),x_2(T^2),\ldots,x_N(T^2)\right\}\right),
$$ where $\overline{co}$ is the closed convex hull of the values in the brackets (defined in \eqref{e-co} below).


\item[{\bf P3)}]
{\bf Convergence to cluster points.}
Every solution  $x(t)$ converges for $t\to +\infty$
to a cluster point, namely to a point
$x^{\infty}=(x^\infty_1,\ldots,x^\infty_N)\in\R^{nN}$, $x^\infty_i\in\R^n$, such that for every $i\in V$,  for every $j\in N_i(x^{\infty})$ it holds $x^\infty_i=x^\infty_j$. Every set of agents with coincident states is said to be a cluster.
 \end{itemize}

Properties  P1), P2), P3) will be discussed for both metric and topological models. Many examples will show the richness of possible behaviors, depending on the chosen notion of solution. 
Indeed, the following theorem summarizes the results that we prove in the next sections: the proof scheme is summarized in Table~\ref{tab:scheme}.

\begin{theorem}[Properties of solutions]\label{t-prop}
\begin{itemize}
    \item [(i)]
{\bf Metric bounded confidence model.} 
     Caratheodory and Krasovsky solutions to \eqref{eq:mbc}  satisfy properties P1)-P2)-P3).
        
 \item[(ii)]    
    {\bf Topological  bounded confidence model.}
    Caratheodory and  Krasovsky solutions to \eqref{eq:tbc}  satisfy property P2) and may not satisfy properties P1) and P3).
    
\item[(iii)]    
    {\bf Topological  bounded confidence model with one neighbor.} In addition to the previous case,  Caratheodory solutions to \eqref{eq:tbc} with $\kappa=1$  satisfy property P3). 
\end{itemize}
\end{theorem}

\begin{table}
    \centering
    \begin{tabular}{|>{\centering}p{6cm}  |>{\centering}p{25mm}  |>{\centering}p{25mm}  |>{\centering}p{25mm}| }
 \hline
 & P1  & P2 & P3  \tabularnewline

 \hline
 Metric    Caratheodory      & Yes    Prop.~\ref{prop:metric-average} & Yes Prop.~\ref{p-contractive} & Yes Prop.~\ref{p-Fclust}  \tabularnewline

\hline
Metric  Krasovsky     & Yes  Prop.~\ref{prop:metric-average} & Yes Prop.~\ref{p-contractive} & Yes Prop.~\ref{p-Fclust}   \tabularnewline

\hline
Topological   Caratheodory      &  No \   Ex.~\ref{ex:t-soluzioni} & Yes  Prop.~\ref{p-contractive} &  No Ex.~\ref{ex:C-eq-no-cluster}     \tabularnewline

\hline
Topological  Krasovsky      & No \  Ex.  \ref{ex:t-soluzioni} & Yes  Prop.~\ref{p-contractive} & No Ex.~\ref{ex:C-eq-no-cluster}    \tabularnewline

\hline
Topological   Caratheodory  $\kappa=1$    &  No \   Ex.~\ref{ex:t-soluzioni} & Yes  Prop.~\ref{p-contractive} &   Yes Prop.~\ref{p-Fclust}    \tabularnewline

\hline
Topological  Krasovsky    $\kappa=1$  & No \  Ex.  \ref{ex:t-soluzioni} & Yes  Prop.~\ref{p-contractive} &  No Ex.~\ref{ex:K-eq-no-cluster-k=1}   \tabularnewline

\hline
\end{tabular}
    \caption{This table summarizes, for the reader's convenience, where in paper the main properties of the bounded confidence models are proved or disproved.}
    \label{tab:scheme}
\end{table}


% ...




Some additional facts are easy to observe.
\begin{remark} 
[Structure of cluster points, metric case] Note that for the metric bounded confidence model, different values assumed by the components of a cluster point $x^{\infty}$ are at a distance greater than or equal to one. Actually, they can be at distance precisely one, as shown by Example \ref{ex1} below. 
\end{remark}

\begin{example}[Clusters at distance 1] \label{ex1}
Consider the system \eqref{eq:mbc} with $n=2, \, N=3$ and initial condition $\overline{x}= \left ((0,0), (1,\frac{1}{3}), (1,-\frac{1}{3}) \right)$. There is a unique Krasovsky (thus also Caratheodory) solution starting at $\overline x$ which converges to the cluster point $x^\infty =\left((0,0),(1,0),(1,0)\right)$. Note that the distance between the first two agents in $x^\infty$ is precisely one.
\end{example}

\begin{remark}[Non-exclusive dependence of the asymptotic state on the initial data]
A desirable property for solutions of any system is that the asymptotic state depends on the initial datum only. Bounded confidence models fail to have this property, because  different solutions starting from the same initial condition can have different  asymptotic states. 
This is the case for Caratheodory (and a fortiori also for Krasovsky) solutions, as shown by Example \ref{ex:t-soluzioni} below.
\end{remark}

Complete proofs of the properties of the metric bounded confidence model can be found in \cite{piccoli_rossi_2020}.
Here we recall the main ideas in order to  compare  metric and topological cases.  In fact, the topological bounded confidence model turns out to be rather different and more complex to characterize. A key reason, already pointed out, is that interactions are not symmetric. As a consequence, even a characterization of equilibria is not evident. 
Here, we construct a Lyapunov function and prove convergence to cluster points in the case $\kappa=1$ only. In this case, we can also characterize the configuration of the network induced by \eqref{eq:tbc} at any time, which is a directed pseudo-forest with a cycle of length 2 in each connected component (Proposition~\ref{p-topology-case-k=1}). In the general case, counterexamples show that convergence to cluster points cannot be expected (Examples~\ref{ex:C-eq-no-cluster}-\ref{ex:K-eq-no-cluster-k=1}).

This picture shows the theoretical interest of these models and the long way to go to fully understand them. 

\section{Generalized solutions: definitions and basic facts}\label{sec:BD}
In this article, we denote by  $\lambda^m$ the Lebesgue measure on $\R^m$. 
For $x\in\R^m$, $B(x,r)$ is the ball of radius $r>0$
centered at $x$ and $B(r)=B(0,r)$ is the ball centered at the origin. The Euclidean norm in $\R ^m$ is denoted by $\|\cdot \|$.
%
Given an embedded manifold $M\subset\R^m$,
the symbol $\partial M$ denotes the topological boundary.
Given $A\subset \R^m$, we denote by $int(A)$  its interior, by $\overline A$ its closure    and we set 
\begin{equation}\label{e-co}
co(A)=\left\{\sum_{i=1}^\ell \alpha_i x_i : \ell\in\N, \alpha_i\in [0,1], \sum_{i=1}^\ell\alpha _i=1, x_i\in A \right\}
\end{equation}
the convex hull of $A$, and denote by $\overline{co}(A)$ its closure.\\

We denote by $AC([0,T],\R^m)$ the space of absolutely continuous functions on a time interval $[0,T]$. Recall that every absolutely continuous function is
differentiable for almost every time, i.e.\ except for times on a set of zero Lebesgue measure.
%
 We also introduce the following: \begin{definition}[Stratified set]\label{def:strat}
 A set $\Gamma\subset\R^m$, $\Gamma=\cup_{i=1}^{m_\Gamma} M_i$, 
 with $m_\Gamma\in\N\cup\{+\infty\}$ and $M_i$ being ${\cal C}^1$ embedded manifold of dimension $n_i\leq m$, is stratified if:
 \begin{itemize}
 \item[i)] The family $M_i$ is locally finite: given a compact $K$, it holds $K\cap M_i\not=\emptyset$ only for finite many $i$.
 \item[ii)] for $i\not=j$  it holds $M_i\cap M_j=\emptyset$,  and if $M_i\cap \partial M_j\not=\emptyset$ then $M_i\subset \partial M_j$ and $n_i<n_j$.
 \end{itemize}
  We call $\max_i n_i$ the dimension of the stratified set $\Gamma$.
 \end{definition}
 \begin{remark}
 For simplicity we used the definition of topological stratification, even
 if the examples we consider  admit Whithney 
 stratification.
 We refer the reader to \cite{marigo_piccoli_2002,PS00,Suss90} for a discussion of the different concepts
 and the role played for discontinuous ODEs and optimal feedback control.
 \end{remark}

An autonomous ODE is written as:
\begin{equation}\label{eq:ODE}
\dot{x}(t)=g(x(t))
\end{equation}
where $x\in\R^m$ and $g:\R^m\to \R^m$ is a measurable and locally bounded
function (defined at every point). The different concepts of solution
will be discussed in the next Section~\ref{s-sols}.

A multifunction on $\R^m$ is a function $H:\R^m\to \PP(\R^m)$,
with $\PP(\R^m)$ being the powerset of $\R^m$, i.e.\ the set of subsets of $\R^m$.
Given a multifunction $H$, one can consider the differential inclusion:
\begin{equation}\label{eq:diff-incl}
\dot{x}(t)\in H(x(t)).
\end{equation}
A solution is an absolutely continuous function $x(\cdot)$ which satisfies
\eqref{eq:diff-incl} for almost every $t$. 

We define the Hausdorff distance $d_H$ on the powerset of $\R^m$ 
as follows: given $x\in\R^m$ and $A,B\subset\R^m$ we set
$d(x,A)=\inf\{d(x,y):y\in A\}$ and 
$d_H(A,B)=\sup \{d(x,A),d(y,B):x\in B,y\in A\}$.
A multifunction $H$ is continuous if it is continuous for the Hausdorff
distance, while $H$ is upper semicontinuous at $x$ if for every $\epsilon>0$
there exists $\delta>0$ such that $H(y)\subset H(x)+B(\eps)$
for every $y$ with $\| x-y\|<\delta$.\\  
A continuous multifunction $H$ is also upper semicontinuous. It is well known that if $H$ is upper semicontinuous 
with compact convex values, then the corresponding differential inclusion \eqref{eq:diff-incl} admits solutions (locally in time) for
every initial condition, see \cite{AubinCellina}. More precisely, we
state the following fact.

\begin{prop} \label{p-exAC} Assume that the multifunction $H$ in \eqref{eq:diff-incl} is upper semicontinuous and, for every $x\in\R^m$, $H(x)$ is a nonempty, compact and convex subset of $\R^m$. Then, for every initial condition $x_0$ there exists a local solution to \eqref{eq:diff-incl}.
\end{prop}

\subsection{Solutions to discontinuous ordinary differential equations} \label{s-sols}
Given the ODE \eqref{eq:ODE} with $g$ discontinuous, it is convenient to define the associated   Krasovsky multifunction,
 defined as:
\begin{equation}\label{eq:Kras-multi}
K(x)=\bigcap_{\delta>0}
\overline{co} \{g(y):y\in (x+B_\delta)\}.
\end{equation}
Similarly, the 
Filippov multifunction is defined as:
\begin{equation}\label{eq:Fil-multi}
F(x)=\bigcap_{\delta>0}\bigcap_{\lambda^m(N)=0}
\overline{co} \{g(y):y\in (x+B_\delta\setminus N)\}.
\end{equation}
Many definitions of solutions for \eqref{eq:ODE} are then available, most of which coincide when $g$ is sufficiently regular (e.g.\ locally Lipschitz). 
We summarize in the following definition the concepts we are considering in the rest of the paper.
\begin{definition}[Notions of solution]\label{def:ODE-sol}
Given the ODE \eqref{eq:ODE} and $T>0$ we define the following:
\begin{enumerate}

\item A {\bf classical solution} is a differentiable function $x:[0,T]\to \R^m$ that  satisfies \eqref{eq:ODE} at every time $t\in (0,T)$. At $0$ and at $T$ the equation must be satisfied
with  one-sided derivatives.

\item A {\bf Caratheodory solution} is an absolutely continuous function $x:[0,T]\to \R^m$  which satisfies \eqref{eq:ODE} at almost every time $t\in [0,T]$.

\item A {\bf Krasovsky solution} is an absolutely continuous function $x:[0,T]\to \R^m$, which satisfies: 
\[
\dot{x}\in K(x(t))
\]
for almost every time $t\in [0,T]$, with $K$ given by \eqref{eq:Kras-multi}.

\item A {\bf Filippov solution} is an absolutely continuous function $x:[0,T]\to \R^m$, which satisfies: 
\[
\dot{x}\in F(x(t))
\]
for almost every time $t\in [0,T]$, with $F$ given by \eqref{eq:Fil-multi}.
\end{enumerate}
We denote the sets of  classical, Caratheodory, Filippov and Krasovsky solutions with $\mathcal Cl$, $\mathcal Ca$, $\mathcal F$ and $\mathcal K$ respectively.

\end{definition}

The concept of classical solution is not used for discontinuous ODEs, because of general lack of existence.
In the following examples we show that both models may not admit a classical solution for some initial condition. In these examples and later in this paper, it will be convenient to denote the vector fields defined by the right-hand sides of the metric model \eqref{eq:mbc} and the topological model \eqref{eq:tbc} by $f^m$ and $f^t$ respectively.


\begin{example}[Non-existence of classical solutions, metric]
Let $N=3,\ n=1$,  $a\equiv 1$ and consider point $\overline x=(-\frac{1}{2},0,\frac{1}{2})$.
Let $f^m$ be the vector field defined by the right-hand side of \eqref{eq:mbc}. 
We have $f^m(\overline x)=(\frac{1}{2},0,-\frac{1}{2})$, in fact agents $1$ and $3$ do not communicate in this configuration. 
As soon as $t>0$ agents $1$ and $3$ start communicating as $|x_3(t)-x_1(t)|<1$ for $t>0$. Then we have that $\lim _{t\to 0^+}f^m(x(t))= (\frac{3}{2},0,-\frac{3}{2})$ which is different from $f^m(\overline x)$.
This proves that a classical solution issuing from $\overline x$ does not exist. 
If we take the initial condition $(-\frac{2}{3},0, \frac{2}{3})$, a classical solution exists until the state $\overline{x}$ is reached, but cannot be continued up to $+\infty$.


\end{example}

\begin{example}[Non-existence of classical solutions, topological]
Consider \eqref{eq:tbc}
with  $N=4$, $n=2$, $\kappa=1$, $a\equiv 1$ and 
the initial condition $\bar{x}=((-1,0), (0,0),(1,0),(1-\epsilon, \sqrt{1-\epsilon^2}))$
with $0<\epsilon<\frac{1}{2}$.
Then $N^t_1(\bar x)=\{2\}$, $N^t_2(\bar x)=\{1\}$,
$N^t_3(\bar x)=\{2\}$, $N^t_4(\bar x)=\{3\}$.
Therefore $\dot{x}_3-\dot{x}_2=0$
and $( \dot{x}_4-\dot{x}_3)\cdot ( x_4-x_3) <0$, thus for all positive times it holds
$N^t_3(x(t))=\{4\}$ and there is no classical solution.
\end{example}

Caratheodory solutions are among the ones commonly used, as they are equivalent to solutions in the integral form:
\[
x(t)=x(0)+\int_0^t g(x(s))\,ds.
\]
Existence theorems for Caratheodory solutions are far from trivial, as we will see in Section~\ref{s-sols}.

The concepts of Filippov and Krasovsky solutions are often used
to deal with general discontinuous ODEs. They have the advantage of being based on the well-developed
theory of differential inclusions \eqref{eq:diff-incl}, see \cite{AubinCellina,Filippov}.
In particular, we have the following proposition, see \cite{AubinCellina}.

\begin{prop}[Local existence]\label{prop:existence-Fil}
Consider an ODE \eqref{eq:ODE} with $g$ measurable and locally bounded.
Then the corresponding Krasovsky and Filippov multifunctions $K$ and $F$ defined by \eqref{eq:Kras-multi} and  \eqref{eq:Fil-multi} respectively, are  upper semicontinuous with nonempty,
compact and convex values. Thus, the differential inclusions $\dot{x}\in K(x)$ and $\dot{x}\in F(x)$
admit local solutions for every initial condition. 
\end{prop}

Among  solutions a special role is played by {\em equilibrium solutions}, whose notion should of course be adapted to the chosen concept of solution. More precisely, we give the following definition.
\begin{definition}
We call $\overline x\in \R^m$ an equilibrium with respect to classical (respectively Caratheodory,  Krasovsky, Filippov) solutions, if the function $\phi(t)=\overline x$ is a classical  (respectively Caratheodory,  Krasovsky, Filippov) solution.  
\end{definition}
We remark that $\overline x$ is an equilibrium with respect to classical and Caratheodory solutions if and only if $f(\overline x)=0$. This fact implies that Caratheodory and classical equilibria coincide. 
Instead, $\overline x$ is an equilibrium with respect to Krasovsky (respectively Filippov) solutions if and only if 
$0\in K(\overline x)$ (respectively $0\in F(\overline x)$).

\subsection{Inclusions between sets of solutions}
\label{s-inclusions}

In this section, we study the inclusions between the different concepts of solutions introduced above.
We first recall the standard inclusions between solutions, that do not depend on the specific structure of  \eqref{eq:general}. The proof is omitted, as it directly follows from definitions.

\begin{prop}[Solution sets]
The following inclusions among sets of solutions hold true:
$\mathcal {C} l \subseteq  \mathcal {C}a \subseteq  \mathcal K$
and $\mathcal {C} l\subseteq \mathcal F\subseteq \mathcal K$.
\end{prop}

We now prove that in the specific case of dynamics \eqref{eq:mbc} and \eqref{eq:tbc} the sets of Filippov and Krasovsky solutions actually coincide.  
In view of this result, in the rest of this paper we will no longer distinguish between Krasovsky and Filippov solutions and we will simply refer to them as to Krasovsky solutions.
The proof is based on the following fact.

\begin{lemma}[Lemma 2.8 in \cite{Hajek}]\label{Lemma:Hajek}
Let $f: \R ^m\to \R^m$ be such that: 
\begin{itemize}
\item[(i)] there exist $M_\alpha\subseteq \R^m, \alpha \in \mathcal A$, such that $\cup_\alpha {M_\alpha}=\R^m$, 
 $M_{\alpha}\cap M_{\beta }= \emptyset$ for all $\alpha, \beta\in \mathcal A, \alpha\not =\beta $, and $M_{\alpha}\subseteq \overline{\text{int} (M_{\alpha})}$ for all $\alpha\in \mathcal A$,

\item[(ii)] there exist $f_{\alpha}:\R^m\to \R ^m$ continuous such that $f(x)=f_{\alpha}(x)$ for all $x\in M_{\alpha}$ and for all $\alpha\in \mathcal A$.
\end{itemize}
Then $\mathcal K=\mathcal F$ for \eqref{eq:ODE}. 

\end{lemma}

\begin{prop}[Krasovsky and Filippov solutions coincide] \label{prop:K=F}
For the metric model \eqref{eq:mbc} and the topological model \eqref{eq:tbc}, it holds $\mathcal K=\mathcal F$.  \end{prop}

\begin{proof}
The system \eqref{eq:general} can be written in standard form
\eqref{eq:ODE} by setting $m=nN$,
$x=(x_1,\ldots,x_N)\in\R^{nN}$,
$f=(f_1,\ldots,f_N)$ with $f_i:\R^n\to\R^n$ given
by the right-hand side of \eqref{eq:general}.

We start considering the {\em metric bounded confidence model} \eqref{eq:mbc}.
Given $i,j\in V$, $i\not= j$, we define
the subset of $\R^{nN}$:
\begin{equation}
\Delta_{ij}^m=\{(x_1,\ldots,x_N)\in \R ^{nN}: \|x_i-x_j\|=1\},\quad
\end{equation}
and the union of such subsets as:
\begin{equation} 
\Delta^m=\cup_{i,j:i\not= j}\Delta_{ij}^m.
\end{equation}
The set $\Delta ^m$ is the set of points
at which the right-hand side of \eqref{eq:mbc} fails to be continuous.  $\R ^{nN}$ is the disjoint union of  $p=2^{\binom{N}{2}}$  sets such that $f^m$ restricted to each of them is continuous. We can enumerate these sets by 
starting with $M_1=\{x\in \R^{nm}: \|x_i-x_j\|<1 \, \forall i,j\in V, i\not =j \}$, $M_2=\{x\in \R^{nm}: \|x_i-x_j\|<1 \, \forall i,j\in V \, \text{except \, for \, } \|x_1-x_N\|\leq 1 \}$ and finishing with 
$M_{p}=\{x\in \R^{nm}: \|x_i-x_j\|\geq 1\,  \forall i,j\in V \}$.  Since $M_1,...,M_{p}$  and $f^m$ satisfy  the assumptions of Lemma \ref{Lemma:Hajek}, then $\mathcal K=\mathcal F$ for  \eqref{eq:mbc}.

An analogous argument can be repeated for the {\em topological bounded confidence model} \eqref{eq:tbc}.
In this case we denote by 
\begin{equation}
\Delta_{ijh}^t=\{(x_1,\ldots,x_N)\in \R ^{nN}: \|x_j-x_i\|=\|x_h-x_i\|\}
\end{equation}
and by
\begin{equation} 
\Delta^t=\cup_{i,j,h:i\not= j\not = h\not =i }\Delta_{ijh}^t.
\end{equation}
Remark that the right-hand side of \eqref{eq:tbc} is discontinuous on a subset of $\Delta ^t$. Also in this case, $\R ^{nN}$ is the disjoint union of  a finite number of  sets delimited by the $\Delta_{ijh}^t$'s such that $f^t$ restricted to each of them is continuous.
\end{proof}

The next examples  show that the inclusions $ \mathcal {C}l \subseteq \mathcal {C}a \subseteq  \mathcal K$ are proper for both dynamics. It also shows that Caratheodory and Krasovsky solutions starting from the same initial condition may converge to different equilibria.

\begin{example}[Proper inclusions between solution sets, metric]\label{ex:m-soluzioni}
Consider \eqref{eq:mbc} with $N=3, n=1$, $a \equiv 1$ and the initial condition $\overline x=(-\frac{1}{3}, 0,1)$. 
Note that $\overline x$ is a discontinuity point of $f^m(x)$ as $x_3-x_2=1$. 
We have $f^m(\overline x)=(\frac{1}{3},-\frac{1}{3},0)$.  In fact agent $2$ and $3$ do not communicate. There exists a unique classical solution
$x_{\cal C} (t)=(-\frac{1}{6}-\frac{1}{6}e^{-2t}, -\frac{1}{6}+\frac{1}{6}e^{-2t}, 1)$
which converges to the point $(-\frac{1}{6}, -\frac{1}{6},1)$. 

%
If we consider Caratheodory solutions, we note that there exists one more solution,  that starts following the limit value of the vector field as $x_3-x_2\to 1^-$, namely $f^{m-}(\overline x)=(\frac{1}{3}, \frac{2}{3}, -1)$: this Caratheodory solution behaves as if agents $2$ and $3$ communicate. Its 
expression is 
$x_{\mathcal{C}a} (t)= (\frac{1}{9}e^{-3t}-\frac{2}{3}e^{-t}+\frac{2}{9}, -\frac{2}{9}e^{-3t}+\frac{2}{9}, \frac{1}{9}e^{-3t}+\frac{2}{3}e^{-t}+\frac{2}{9})$ 
and it converges to $(\frac{2}{9}, \frac{2}{9}, \frac{2}{9} )$. 

We finally consider Krasovsky solutions. Besides the ones already obtained there exists a solution that slides on the discontinuity plane $\pi : x_3-x_2=1$. In fact admissible directions ${\tilde f}^{m}$
at the points of $\pi$ belong to the set
$$Kf(x)=\setdef{\alpha (x_2-x_1,
      1+x_1-x_2,
      -1)+(1-\alpha) (x_2-x_1,     x_1-x_2,
      0)}{ \alpha \in  [0,1]}.$$
Since the normal vector to $\pi$ is $v_{\bot}=(0,-1,1)$, we have that
$v_\bot \cdot \dot x=- 2 \alpha+x_2-x_1$
is equal to zero if $\alpha=\frac12(x_2-x_1).$ Namely, the Krasovsky
solution corresponding to this $\alpha$ does not exit the discontinuity plane but slides on it. In fact the sliding solution keeps $x_3$ and $x_2$ at distance 1 as $\dot x_3-\dot x_2=0$. The solution can stay on the discontinuity for arbitrarily long time: if it remains there forever, then it converges to the point $(-\frac{1}{9}, -\frac{1}{9}, \frac{8}{9})$. Other Krasovsky solutions may exit $\pi$ at arbitrary times $T$ in two different ways: either agents $2$ and $3$ influence each other and the solution converges to  $(\frac{2}{9}, \frac{2}{9},\frac{2}{9})$, or they stop interacting at all and the solution converges to $(x^*,x^*,x_3(T))$ with $x^*=\frac13-\frac{x_3(T)}2$. 
\end{example}

\begin{example}[Proper inclusions between solution sets, topological]\label{ex:t-soluzioni}
We consider \eqref{eq:tbc} with $N=3$, $n=1$, $\kappa=1$, $a\equiv 1$  and the initial condition
$\overline x=(0,-1,1)$.
We remark that the vector field defined by the right-hand side of \eqref{eq:tbc} is discontinuous at $\overline x$ as it belongs to the plane $\pi: (x_1-x_2)-(x_3-x_2)=0$.
In order to have the equation satisfied at $t=0$ classical solutions must satisfy the equations
$$
\begin{cases}
     \dot x_1 = x_2-x_1 \\
     \dot x_2 = x_1-x_2\\
     \dot x_3 = x_1-x_3.\\
\end{cases}
$$
There is a unique classical solution starting from $\overline x$ and it converges to the point $\left (-\frac{1}{2}, -\frac{1}{2},-\frac{1}{2}\right) $.
In fact $\dot x_1+\dot x_2=0$, then 
$x_1(t)+x_2(t)=\bar x_1+\bar x_2=\lim _{t\to +\infty }[x_1(t)+x_2(t)]=-1$. 

We now observe that there is a Caratheodory solution that does not satisfy the equations at $t=0$ but does for $t>0$, namely the solution of the equations 
$$
\begin{cases}
     \dot x_1 = x_3-x_1 \\
     \dot x_2 = x_1-x_2\\
     \dot x_3 = x_1-x_3.\\
\end{cases}
$$
This Caratheodory solution converges to the equilibrium point $\left (\frac{1}{2}, \frac{1}{2},\frac{1}{2}\right) $.

Finally we  have a Krasovsky solution starting at $\overline x$ that slides on the plane $\pi$.
In fact, if we denote by $f^{t-}(x)$ and $f^{t+}(x)$ its limit values as $x$ approaches 
the plane $\pi$ from the negative and positive sides respectively, we have
$f^{t-}(x)=(x_2-x_1, x_1-x_2, x_1-x_3)$ and 
$f^{t+}(x)=(x_3-x_1, x_1-x_2, x_1-x_3)$. 
We can then compute the Krasovsky set-valued map on $\pi$:
$$
Kf(x)=\{(\alpha x_2+(1-\alpha )x_3-x_1, x_1-x_2, x_1-x_3), \alpha\in[0,1]\}.
$$
By posing $\alpha=\frac{1}{2}$  we obtain the admissible direction 
$f^t_{1/2}(x)=(0, x_1-x_2,x_1-x_3)$ which is parallel to $\pi$. This implies that there is a Krasovsky solution starting from $\overline x$ and sliding on $\pi$, namely $x_1(t)=0, x_2(t)=-e^{-t}, x_3(t)=e^{-t}$, which converges to the origin.
\end{example}

\subsection{P1) Average preservation}

In this section, we discuss property P1), that is preservation of the average value of the agents. We prove that P1) is satisfied for metric interaction models, while this is not the case for topological interactions. This is one more consequence of the lack of symmetry of topological interactions.

\begin{prop}[Average preservation]\label{prop:metric-average}
Caratheodory and Krasovsky solution of \eqref{eq:mbc} have property P1). \end{prop}
The proof of Proposition~\ref{prop:metric-average} can be found in \cite{frasca} in the case $n=1$ and in \cite{piccoli_rossi_2020} in the general case.
The same property does not hold for solutions of the topological bounded confidence model, by the following example.

\begin{example}[{\em Example \ref{ex:t-soluzioni}, continued}]
Let $x(t)$ be the unique classical  solution starting from the point $(0,-1,1)$. Observe that $x(t)$ is such that $x_{ave}(0)=0$, but the limit of $x(t)$ for $t\to+\infty$  is $(-\frac{1}{2},-\frac{1}{2},-\frac{1}{2})$, so that 
$\lim _{t\to +\infty} x_{ave} (t)=-\frac{1}{2}$.
\end{example}

\subsection{P2) Contractivity of the support}

In this section, we prove that the support of solutions (in any sense given above) is weakly contractive. This is a well-known property of Caratheodory solutions for bounded confidence models, see e.g.\ \cite{blondel2}. The proof of such property for Krasovsky solutions for the metric model \eqref{eq:mbc} on the real line can be found in \cite[Prop.~3.iii]{frasca}.

We will give a general proof for Krasovsky solutions in any dimension, both for the metric \eqref{eq:mbc} and topological models \eqref{eq:tbc}.
The proof is similar to the one provided in \cite{piccoli_rossi_2020},
thus we provide a sketch only.


\begin{prop}[Contractivity of the support]\label{p-contractive} Let $x(t)=(x_1(t),x_2(t),\ldots,x_N(t))$ be a solution to either \eqref{eq:mbc} or \eqref{eq:tbc},
in any of the senses given in Definition \ref{def:ODE-sol}. Assume
$a:[0,+\infty [\to [0,+\infty[$ continuous 
and $0\leq T^1<T^2$, then
\begin{equation}
\overline{co}\left(\left\{x_1(T^1),x_2(T^1),\ldots,x_N(T^1)\right\}\right)\supseteq \overline{co}\left(\left\{x_1(T^2),x_2(T^2),\ldots,x_N(T^2)\right\}\right).\label{e-inclusion}
\end{equation}
\end{prop}
\begin{proof}  Let $x(\cdot)$ be a Krasovsky solution. Define $X(t):=\overline{co}\left(\left\{x_1(t),x_2(t),\ldots,x_N(t)\right\}\right)$ and
$$A(T^1):=\left\{T\in (T^1,+\infty)\mbox{~~s.t.~~} X(T^1)\not\supseteq X(T)\right\}.$$
We claim that $A(T^1)$ is empty, which implies 
\eqref{e-inclusion}. Otherwise,
by contradiction, we can define $T^3=\inf A(T^1)\geq T^1$. Following the same argument as in \cite{piccoli_rossi_2020}, we can prove the following:\\
\noindent {\bf Claim a)} It holds either $\inf(A(T^1))=T^1$ or 
$\inf(A(T^3))=T^3$.\\
Without loss of generality, we can assume
$T^1=0$ or $T^3=0$, thus $\inf(A(0))=0$.
Let $t_k\searrow 0$ be such that $x_i(t_k)\not\in X(0)$ for a fixed $i\in V$, then by continuity of the trajectory it holds 
$\bar x_i:=x_i(0)\in \partial X(0)$, where $\partial$ indicated the topological boundary.
Since $X(0)$ is a convex polyhedron, it
is supported by a finite number of hyperplanes at $\bar x_i$, thus, by possibly passing to subsequences, we can find a unitary vector $\nu$
such that
$(x_i(t_k)-\bar x_i)\cdot \nu>0$ for all $t_k$. 
Moreover, it holds $(x_j(0)-\bar x_i)\cdot \nu\leq 0$ for all $j\in V$.\\
Now, define $\phi_j(x)=(x_j-\bar x_i)\cdot \nu$, and $\phi(x):=\max_{j\in V}\phi_j(x)$. Observe that $\phi(x(0))\leq 0$ and $\phi(x(t_k))>0$.
We can apply Danskin Theorem \cite{danskin} to $\phi$, thus even though  $\phi$ may be not differentiable, it admits directional derivatives. 
Denote by $h=(h_1,\ldots,h_N)$
the displacement, 
then by applying Danskin formula, the directional derivative $D_h$ along $h$ is given by
\begin{eqnarray*}
D_h \phi(x)&=&\max_{j\in A_i(x)}
\sum_{k=1}^N h_k \cdot \nabla_{x_k} \phi_j(x)=\max_{j\in A_i(x)}
 h_j \cdot \nabla_{x_j} \phi_j(x)=\max_{j\in A_i(x)}
 h_j \cdot \nu, \end{eqnarray*}
where $A_i(x)$ is the set of indexes $j\neq i$ realizing the maximum in the definition of $\phi(x)$.
Since $D_h \phi(x)$ is always defined and $\dot x(t)$ exists for almost every $t\in(0,T)$, 
then also  $\dot \phi(x(t))$ exists for almost every $t\in(0,T)$. Moreover, by direct computation, we get:
\begin{eqnarray}\label{e-derdir}
\lim_{\tau\to 0}\frac{\phi(x(t+\tau))-\phi(x(t))}\tau&=&
\lim_{\tau\to 0}\frac{\phi(x(t)+\tau\dot x(t)+o(\tau))-\phi(x(t))}\tau\\
&=&\lim_{\tau\to 0}\frac{\phi(x(t))+\tau D_{\dot x(t)}\phi(x(t))+o(\tau)-\phi(x(t))}\tau=D_{\dot x(t)}\phi(x(t)),\nonumber
\end{eqnarray}
thus it holds:
\begin{equation} \label{e-dphi}
    \dot \phi(x(t))=\max_{j\in A_i(x(t))}\dot \phi_j(t)=\max_{j\in A_i(x(t))}\dot x_j(t)\cdot \nu.
\end{equation}
Now, if $x(\cdot)$ is differentiable at $t$, and $j\in A_i(x(t))$, then for every $k\neq j$ we have
\begin{equation}
(x_k(t)-x_j(t))\cdot \nu=(x_k(t)-\bar x_i)\cdot \nu+(\bar x_i-x_j(t))\cdot \nu=\phi_k(x(t))-\phi_j(x(t))\leq \phi(t)-\phi(t)=0.\label{e-jineg}
\end{equation}
Since $x(\cdot)$ is a Krasovsky solution, 
there exist $b_{jk}\geq 0$ such that $\dot x_j=\sum_{k=1}^N b_{jk} a(\|x_k-x_j\|)(x_k-x_j)$.
Substituting this expression in \eqref{e-dphi}, we get:
$$\dot \phi(x(t))=\max_{j\mbox{~s.t.~}\phi(t)=\phi_j(t)}\sum_{k=1}^N b_{jk} a(\|x_k-x_j\|)(x_k-x_j)\cdot \nu\leq 0.$$
This contradicts the fact that $\phi(x(0))=0$ and $\phi(t_k)>0$. Thus \eqref{e-inclusion} holds,
for the Krasovsky solution $x(\cdot)$.
Since the proof holds for every Krasovsky solution, by recalling the inclusions of Section~\ref{s-inclusions},
the statement holds for any definition of solution. \end{proof}

\section{Existence and uniqueness of solutions}\label{sect:solutions}

In this section, we study existence and uniqueness of solutions, both for the metric and the topological models.

\subsection{Existence of solutions}
\begin{prop}[Existence of Krasovsky solutions]
For any initial condition, equations \eqref{eq:mbc} and \eqref{eq:tbc} admit a Krasovsky solution defined on $[0,+\infty )$.
\end{prop}
\begin{proof}
For both \eqref{eq:mbc} and \eqref{eq:tbc} the right-hand side is locally bounded. The local existence of Filippov solutions then follows from Proposition~\ref{prop:existence-Fil}. By Proposition~\ref{prop:K=F}, the sets of Krasovsky and Filippov solutions coincide, then local  Krasovsky solutions also exist. 
Proposition  \ref{p-contractive} guarantees that solutions are bounded, then they  can  be continued on $[0,+\infty )$ by standard arguments. 
\end{proof}
In general, Krasovsky solutions are not unique, as already shown in Example \ref{ex:t-soluzioni}.
In the following proposition we state the existence of Caratheodory solutions for both metric and topological bounded confidence models. The proof  for the metric model in the case $n=1$ was first given in \cite{blondel2} and then generalized to any $n$ in \cite{piccoli_rossi_2020}. The proof for the topological case with $\kappa=1$ is new. 
We conjecture that the result holds for $\kappa>1$ as well, although with a more involved argument that we avoid to develop here. 

\begin{prop}[Existence of Caratheodory solutions]
\begin{itemize}
    \item[(i)]{\bf Metric bounded confidence.}  For any initial condition, equation \eqref{eq:mbc}  admits a Caratheodory solution defined on $[0,+\infty )$.

\item[(ii)]{\bf Topological bounded confidence.}
If $\kappa=1$, then any initial condition \eqref{eq:tbc} admits a Caratheodory solution defined on $[0,+\infty )$.
\end{itemize}
\end{prop}

\begin{proof}
We only consider the topological bounded confidence model with $\kappa=1$. We build a Caratheodory solution as follows. 
For each initial datum $\bar x=(\bar x_1,\ldots,\bar x_N)$, we construct an oriented graph $G$ for which there exists $T>0$
and a curve defined on $[0,T]$
having $G$ as connectivity graph.
For each index $i\in V$ there exists one and only one index, that we denote with $\Gamma(i)$, such that $(i,\Gamma(i))\in G$. This implies that $\dot x_i=a(\|x_{\Gamma(i)}-x_i\|) (x_{\Gamma(i)}-x_i)$ for the whole time interval $[0,T]$. We then need to prove that the corresponding trajectory $(x_1(t),\ldots,x_N(t))$ is indeed a Caratheodory solution for \eqref{eq:tbc}. Remark that one might aim to choose $\Gamma(i)$ as the single element of $N_i^t(\bar x)$, that is the minimal index (in the lexicographic order) among nearest neighbours of $\bar x_i$. In our proof, this is not the case, as one might choose $\Gamma(i)$ not to be the nearest neighbour with minimal index at the initial time, but to be the nearest neighbour for all $t\in(0,T)$.
We then conclude the proof by piecing together Caratheodory solutions on time intervals $[0,T_1],[T_1,T_2],\ldots$ to build a solution on $[0,+\infty)$.\\

Let $\bar x$ be an initial configuration, that is fixed from now on. We build an oriented graph $G$ recursively
by selecting, for each index  $i\in V$, a unique index $\Gamma(i)$ such that $(i,\Gamma(i))\in G$. 
First define:
$$A_i:=\mathrm{argmin}_{j\neq i}(\|\bar x_j-\bar x_i\|),$$
which is the set of agents realizing the minimal distance to $\bar x_i$.\\
The graph $G$ is constructed using
the following algorithm:

\begin{description}
    \item[Step 1)] For all $i$ such that $N_i^t=\{j\}$, define $\Gamma(i):=j$.

    \item[Step 2)] WHILE there exists a pair of indexes $i,j$ such that $i\not\in G$, $j\not\in G$ and $j\in A_i,i\in A_j$
    
    DO: define $\Gamma(i):=j$ and $\Gamma(j):=i$.

    \item[Step 3)] IF there exists $i\notin G$ such that $A_i=\{j_1,\ldots,j_l\}$ and $j_1,\ldots,j_l\in G$
    
    DO: choose
    \begin{equation}
            \Gamma(i)\in 
            \mathrm{argmin}_{j\in \{j_1,\ldots,j_l\}} \psi_i(j)
            \label{e-step3}
    \end{equation}
    where 
    \begin{equation}\label{e-psi}
        \psi_i(l):=(x_l-x_i)\cdot\left(a(\|x_{\Gamma(l)}-x_l\|)(x_{\Gamma(l)}-x_l)-a(\|x_{l}-x_i\|)(x_{l}-x_i)\right)
    \end{equation}
    \item[Step 4)] IF for all $i\in V$ it holds $i\in G$, STOP.
    
    ELSE: Go to Step 3.
\end{description}
Observe that the number of edges of $G$
is increased at each step and is bounded by $N$. Thus, there exists a limit graph $G'$, reached after a finite number of steps. 
We now prove the following claim:
\begin{center}
{\bf (C)}\hspace{1cm}  for all $i\in V$ ~~~~it holds~~~~ $i\in G'$.
\end{center}
To prove {\bf (C)}, assume by contradiction
that there exists $i\notin G'$ and,
by possibly relabeling indexes, $i=1\not\in G'$.
By definition of $G'$, Step 3 does not add edges
to $G'$, in particular no edge to agent $i=1$; thus $A_1$
contains at least one index, that we relabel as $2$, such that $2\not\in G'$. 
If $1\notin A_2$, we can find another index,
relabeled as $3$, such that $3\notin G'$ and so on.
Finally there exists $k\not\in G'$ such that $A_k\cap \{1,2,\ldots,k-1\}\neq \emptyset$.
Possibly reducing the sequence and changing the initial element, we assume $1\in A_k$.\\
Now, if there exist $i,i+1\in \{1,2,\ldots,k\}$ such that $i+1\in A_i$ and $i\in A_{i+1}$, then
we are in contradiction with Step 2.
Therefore, we can assume that for all $i\in\{1,2,\ldots,k-1\}$ it holds $i\not\in A_{i+1}$. Since $i+2\in A_{i+1}$ by construction, we have 
    \begin{equation} \label{e-strict}
        \|x_i-x_{i+1}\|>\|x_{i+1}-x_{i+2}\|\mbox{~~~~for all~}i\in\{1,2,\ldots,k-2\}.
    \end{equation}
Using \eqref{e-strict} and recalling $1\in A_k$ we get
    $$\|x_1-x_2\|>\|x_2-x_3\|>\ldots>\|x_{k-1}-x_k\|>\|x_k-x_1\|.$$
This implies $2\not\in A_1$, achieving a contradiction.
This concludes the proof of claim {(\bf C)}.\\

\noindent
Let $x(\cdot)$ be the curve
satisfying $\dot x_i=a(\|x_{\Gamma(i)}-x_i\|) (x_{\Gamma(i)}-x_i)$, i.e.\ with dynamics
associated to $G$,
with initial condition $\bar x$.
We first show that there exists a time $T>0$ such that $\Gamma(i)\in \mathrm{argmin}_{j\neq i}(\|x_j(t)-x_i(t)\|)$ for all $t\in[0,T]$. 
More precisely, for each $i\in V$, 
and $k\in V\setminus\{i,\Gamma(i)\}$ 
we show that there exists  $T_{ik}>0$ such that
$\|x_i(t)-x_k(t)\|\geq \|x_i(t)-x_{\Gamma(i)}(t)\|$
on $[0,T_{ik}]$. Then it will be sufficient
to define $T= \min_{ik} T_{ik}$ (with the convention that the minimum is $+\infty$ if all $T_{ik}=+\infty$).\\
Now, fix  $i\in V$, 
and $k\in V\setminus\{i,\Gamma(i)\}$.
Notice that if $\|\bar x_i-\bar x_k\|>\|\bar x_i-\bar x_{\Gamma(i)}\|$,
then by continuity there exists $T_{ik}>0$ such that $\|x_i(t)-x_k(t)\|>\|x_i(t)-x_{\Gamma(i)}(t)\|$ for all $t\in[0,T_{ik}]$.
Therefore, from now on, we assume
$$\|\bar x_i-\bar x_k\|\leq \|\bar x_i-\bar x_{\Gamma(i)}\|.$$
By definition of $A_i$, this inequality is indeed an equality, otherwise $\Gamma(i)\not\in A_i$. 
We distinguish two sub-cases:\\
{\bf Case 1)} $\bar x_k=\bar x_{\Gamma(i)}$.\\
{\bf Case 2)} $\bar x_k\neq \bar x_{\Gamma(i)}$.\\
In {\bf Case 1)} there exists a (possibly empty) set of indexes $L:=\{l_1,\ldots,l_r\}$ such that $\bar x_k=\bar x_{\Gamma(i)}=\bar x_{l_m}$, hence $A_k=L\cup\{{\Gamma(i)}\}$. From claim {\bf (C)}, for each $l\in\{k,{\Gamma(i)}\}\cup L$ the neighbour $\Gamma(l)$ is well-defined, thus $\bar x_l=\bar x_{\Gamma(l)}$ by the condition of minimal distance. This in turn implies $\dot x_l\equiv 0$, and similarly for all other indexes in $A_k$. 
Since all indexes $l\in\{k,{\Gamma(i)}\}\cup L$ satisfy such property, it holds $A_l\subseteq \mathrm{argmin}_j (\|x_j(t)-x_l(t)\|)$. Observe moreover that the dynamics does not allow for merging particles, thus the inclusion is indeed an equality. This in turn means that $x_k(t)=x_{\Gamma(i)}(t)$ for all $t>0$, hence we can choose $T_{ik}=+\infty$.\\
Consider now {\bf Case 2)}.
Observe that $\bar x_k\neq \bar x_i$, otherwise $\|\bar x_{\Gamma(i)}-\bar x_i\|\leq \|\bar x_k-\bar x_i\|=0$, which gives $\bar x_{\Gamma(i)}=\bar x_i=\bar x_k$.
We now prove that there exists $T_{ik}>0$ such that
        \begin{equation} \label{e-condki-strong}
        \|x_i(t)-x_{\Gamma(i)}(t)\|< \|x_i(t)-x_k(t)\|\mbox{~~~ for all $t\in(0,T_{ik}]$.}
        \end{equation}
Consider the function
$$\phi_{ik}(t):=\tfrac12 \|x_{\Gamma(i)}(t)-x_i(t)\|^2-\tfrac12 \|x_k(t)-x_i(t)\|^2.$$
Since $\phi_{ik}(0)=0$,
to prove \eqref{e-condki-strong}
it is enough to show $\phi'_{ik}(0)<0$. 
Set $j=\Gamma(i)$, then from \eqref{e-psi}, we get    \begin{eqnarray*}
        \phi'_{ik}(0)&=&(\bar x_j-\bar x_i)\cdot (\dot x_j(0)-\dot x_i(0))-(\bar x_k-\bar x_i)\cdot (\dot x_k(0)-\dot x_i(0))=A_{ijk}-B_{ijk}
    \end{eqnarray*}
    with
    \begin{eqnarray*}
        A_{ijk}&=&\psi_i(j)-\psi_i(k)=
        (\bar x_j-\bar x_i)\cdot (a(\|\bar x_{\Gamma(j)}-\bar x_j\|)(\bar x_{\Gamma(j)}-\bar x_j)-(\bar x_j-\bar x_i)\cdot a(\|\bar x_j-\bar x_i\|)(\bar x_j-\bar x_i)\\
        &&-(\bar x_k-\bar x_i)\cdot
        a(\|\bar x_{\Gamma(k)}-\bar x_k\|)(\bar x_{\Gamma(k)}-\bar x_k)+(\bar x_k-\bar x_i)\cdot(a(\|\bar x_k-\bar x_i\|)(\bar x_k-\bar x_i)),\nonumber\\
        B_{ijk}&=&(\bar x_k-\bar x_i)\cdot(a(\|\bar x_k-\bar x_i\|)(\bar x_k-\bar x_i))-(\bar x_k-\bar x_i)\cdot(a(\|\bar x_j-\bar x_i\|)(\bar x_j-\bar x_i))).
    \end{eqnarray*}
The idea of the decomposition is that the last term in $A_{ijk}$ would correspond to $(\bar x_k-\bar x_i)\dot x_i(0)$ by choosing $k$ as the neighbour of $i$. Thus, $B_{ijk}$ is the corrector given by the actual choice $j=\Gamma(i)$. 
We now show $A_{ijk}\leq 0$ and $B_{ijk}>0$,
which implies $\phi'_{ik}(0)<0$.\\
Consider first $A_{ijk}$. The index $j=\Gamma(i)$ 
was not chosen in Step 1 since $k\neq j$ and $k,j\in A_i$.
If $j=\Gamma(i)$  was chosen in Step 2, then $\Gamma(j)=i$, and
    \begin{eqnarray*}
        A_{ijk}&\leq&-2\|\bar x_j-\bar x_i\|^2 a(\|\bar x_j-\bar x_i\|)+\|\bar x_k-\bar x_i\| \cdot \|\bar x_{\Gamma(k)}-\bar x_k\|
        a(\|\bar x_{\Gamma(k)}-\bar x_k\|)+\\
       &&\|\bar x_k-\bar x_i\| \cdot \|\bar x_k-\bar x_i\|
        a(\|\bar x_k-\bar x_i\|).
    \end{eqnarray*}
By definition of $\Gamma(k)$ we have $\|\bar x_{\Gamma(k)}-\bar x_k\|\leq \|\bar x_i-\bar x_k\|$, and, since we are in {\bf Case 2}, it holds $\|\bar x_j-\bar x_i\|=\|\bar x_k-\bar x_i\|$. 
Recalling that $a(r)$ is non-decreasing, we get $A_{ijk}\leq 0$.
Assume now that  $j=\Gamma(i)$ was chosen in Step 3, then, by construction 
$j\in  \mathrm{argmin}_{\ell} \psi_i(\ell)$,
thus $0\geq \psi_i(j)-\psi_i(k)= A_{ijk}$.\\

We now prove $B_{ijk}>0$. Observe that $j,k\in A_i$, hence both $\bar x_j$ and $\bar x_k$ lay on 
the same circle centered at $\bar x_i$ on the plane
containing $\bar x_i$, $\bar x_j$ and $\bar x_k$. 
Thus $a(\|\bar x_k-\bar x_i\|)=a(\|\bar x_j-\bar x_i\|)$. Since $\bar x_k\neq \bar x_j$ ({\bf Case 2}) and
$\bar x_k \neq \bar x_i$, the amplitude $\alpha$ of the angle $\reallywidehat{\bar x_j \bar x_k \bar x_i}$,
on the plane containing $\bar x_i$, $\bar x_j$ and $\bar x_k$, belongs to $\left(-\pi/2,\pi/2\right)$, thus 
\begin{eqnarray*}
    B_{ijk}&=&(\bar x_k-\bar x_i)\cdot a(\|\bar x_j-\bar x_i\|) (\bar x_k-\bar x_j)=\\
    &&a(\|\bar x_j-\bar x_i\|) \|\bar x_k-\bar x_i\|\, \|\bar x_k-\bar x_j\|\cos(\alpha)>0.
\end{eqnarray*}
From $A_{ijk}\leq 0$ and $B_{ijk}>0$ we 
get $\phi'_{ik}(0)<0$ and we are done.\\

\noindent
We now show that $x(\cdot)$ is a Caratheodory solution of \eqref{eq:tbc}.
If $\Gamma(i)=N_i^t(x(t))$ for all times $t\in(0,T)$, then we are done.
Otherwise, there exists $t\in(0,T)$ such that
$\Gamma(i)\neq k=N_i^t(x(t))$, thus  $\|x_{\Gamma(i)}(t)-x_i(t)\|\geq\|x_k(t)-x_i(t)\|$. Recalling the definition of $T_{ik}$,
we deduce that {\bf Case 1)} holds, thus
$\Gamma(i)\neq k=N_i^t(x(t))$ and $x_{\Gamma(i)}(t)=x_ k(t)$, i.e.\ the indexes $k$ and $\Gamma(i)$ are different but the agents' positions coincide. 
As a consequence, we have
$$\dot x_i=a(\|x_{\Gamma(i)}-x_i\|)(x_{\Gamma(i)}-x_i)=a(\|x_k-x_i\|)(x_k-x_i),$$
and \eqref{eq:tbc} is satisfied.\\

\noindent
We now prove that the trajectory can be prolonged to $[0,+\infty)$. If $T=+\infty$, then we are done.
Otherwise, observe that the trajectory $x(\cdot)$ is compact, due to contractivity of the support proved in Proposition~\ref{p-contractive}, thus
we can use transfinite induction as follows.
Since $\dot x(\cdot)$ is uniformly bounded,  $x(\cdot)$ is a uniformly Lipschitz function of time with Lipschitz constant $L:=\max_{ij}a(\|\bar x_i-\bar x_j\|)$, and $x(T)$ is well-defined. We can apply the same algorithm at time $T$, and find $T_1>0$ such that the trajectory is well-defined on $[T,T_1]$. If $T_1=+\infty$, we are done; otherwise define $T_2<T_3<\ldots$ in the same way and extend the trajectory to $[T_i,T_{i+1}]$. If $T_i=+\infty$ for some $i$ or $\lim_{i\to +\infty}T_i=+\infty$, then we are done. Assume, by contradiction, that $T^*=\lim_{i\to +\infty} T_i<+\infty$. Then $x(T^*)$ is well-defined and using the algorithm we can extend the trajectory beyond $T^*$.\\
Using for $T_i$ the same argument as for $T$,
we have that $x(\cdot)$ is a Caratheodory solution
to \eqref{eq:tbc} on each interval $(T_l,T_{l+1})$.
Since $\{T_i\}$ is a countable set, we are done.
\end{proof}

\subsection{Uniqueness for almost every initial condition}

In this section, we study uniqueness of solutions. Examples \ref{ex:m-soluzioni} and \ref{ex:t-soluzioni} show that Caratheodory solutions are not unique in general (thus neither Krasovsky).
Nevertheless, uniqueness of Krasovsky (and then also Caratheodory) solutions holds  
for almost all initial condition, both for metric and topological models. For the metric case, the result was already given in \cite[Prop.~6.2]{piccoli_rossi_2020}.

We then focus on uniqueness of Krasovsky solutions for almost every initial datum for \eqref{eq:tbc}. We first set
\[
I=\{(i,j,k): i\not= j,
i\not= k, j\not= k\}
\]
and define
\[
{\MM}=\cup_{ijk\in I}\MM_{ijk},\quad
\MM_{ijk}=\{x: 
\|x_i-x_j\|=\|x_j-x_k\|\}.
\]
Notice that $\MM$ contains (in general strictly contains) the set where the right-hand side of  \eqref{eq:tbc} is discontinuous.\\
The main reason for uniqueness is that 
Krasovsky solution cannot enter the manifolds $\MM$ and slide on it, except possibly on a set of codimension two.
We first show this fact for the case $\kappa=1$ for simplicity.

Given $(i,j,k)\in I$,
consider the functions
\begin{equation}\label{eq:def-thetaijk}
 \theta_{ijk}(x) = \|x_j-x_i\|^2- \|x_k-x_i\|^2
\end{equation}
and denote by $\pi_{ijk}$ the subset of the  manifold $\MM_{ijk}$ where 
the right-hand side $f^t$ of  \eqref{eq:tbc} is discontinuous and where $\theta_{jvw}(x)\theta_{khu}(x)$ is different from zero for all $v,w,h,u$ (that is, where the only discontinuity is due to $j$ or $k$). 
We want to prove that $\pi_{ijk}$ cannot be attractive with respect to Krasovsky solutions. 
Fix $\bar x$ a discontinuity point for $f^t$, thus
$\bar x \in \pi_{ijk}$ for some ${(i,j,k)}\in I$   
and either $j\in N^t_i(\bar x)$ or $k\in N^t_i(\bar x)$. 
We denote by $f^{t+}(\bar x)$ and $f^{ t-}(\bar x)$ the limit values of $f^t(x)$ as $x\to \bar x$ and  $\theta_{ijk}(x)>0$ (the neighbor of $i$ is then $k$) and
$\theta_{ijk}(x)<0$ (the neighbor of $i$ is then $j$) respectively.
%
%
We denote by $\Gamma (j)$ the neighbor of $j$ at $\bar x$ and by $\Gamma (k)$ the neighbor of $k$ at $\bar x$. We first denote by $\gamma$ the angle  between the vectors $\bar x_j-\bar x_i$ and $\bar x_k-\bar x_i$ and $l=\| \bar x_j-\bar x_i\|= \| \bar x_j-\bar x_i\|$. If $l=0$, the angle is not uniquely defined, but this plays no role in the following.
Let us compute the two quantities 
 $ \nabla \theta_{ijk} (\bar x)\cdot f^{t+}(\bar x)$ and $ \nabla \theta_{ijk} (\bar x)\cdot f^{t-}(\bar x)$.  We have:
 %
%
\begin{eqnarray*}
\nabla\theta_{ijk} (\bar x)\cdot f^{+t}(\bar x) = & (\bar x_j-\bar x_i)\cdot 
[a(\|\bar x_{\Gamma (j)}-\bar x_j\|)(\bar x_{\Gamma (j)}-\bar x_j)- a(\|\bar x_k-\bar x_i\|)(\bar x_k- \bar x_i)]\\ 
& - (\bar x_k-\bar x_i) \cdot 
[a(\|\bar x_{\Gamma (k)}-\bar x_k\|)(\bar  x_{\Gamma (k)}-\bar x_k)- a(\|\bar x_k-\bar x_i\|)(\bar x_k- \bar x_i)]\\
 = & a(\|\bar x_{\Gamma (j)}-\bar x_j\|)
(\bar x_j-\bar x_i)\cdot 
(\bar x_{\Gamma (j)}-\bar x_j)- a(l)l^2 \cos (\gamma) \\ 
& - a(\|\bar x_{\Gamma (k)}-\bar x_k\|) (\bar x_k-\bar x_i) \cdot 
(\bar x_{\Gamma (k)}-\bar x_k)+ a(l) l^2
\end{eqnarray*}
and 
\begin{eqnarray*}
\nabla \theta_{ijk} (\bar x)\cdot f^{t-}(\bar x) = & (\bar x_j-\bar x_i)\cdot 
[a(\|\bar x_{\Gamma (j)}-\bar x_j\|)(\bar x_{\Gamma (j)}-\bar x_j)- a(\|\bar x_j-\bar x_i\|)(\bar x_j- \bar x_i)]\\ 
 & -(\bar x_k-\bar x_i) \cdot 
[a(\|\bar x_{\Gamma (k)}-\bar x_k\|)(\bar x_{\Gamma (k)}-\bar x_k)- a(\|\bar x_j-\bar x_i\|)(\bar x_j- \bar x_i)]\\
 = & a(\|\bar x_{\Gamma (j)}-\bar x_j\|)
(\bar x_j-\bar x_i)\cdot 
(\bar x_{\Gamma (j)}-\bar x_j)- a(l)l^2 \\ 
& -a(\|\bar x_{\Gamma (k)}-\bar x_k\|) (\bar x_k-\bar x_i) \cdot 
(\bar x_{\Gamma (k)}-\bar x_k)+ a(l) l^2\cos (\gamma). 
\end{eqnarray*}
We have:
$$\nabla\theta_{ijk} (\bar x)\cdot f^{t+}(\bar x)-
\nabla \theta_{ijk} (\bar x)\cdot f^{t-}(\bar x)=-a(l)l^2\cos (\gamma) +a(l)l^2+a(l)l^2 -a(l)l^2\cos (\gamma)=2a(l)l^2 (1- \cos (\gamma)).$$
First of all, we remark that 
$\nabla\theta_{ijk} (\bar x)\cdot f^{t+}(\bar x)-
\nabla \theta_{ijk} (\bar x)\cdot f^{t-}(\bar x)\geq 0$. 
Moreover $\nabla\theta_{ijk} (\bar x)\cdot f^{t+}(\bar x)-
\nabla \theta_{ijk} (\bar x)\cdot f^{t-}(\bar x)=0$ if and only if $\gamma =0$, so that $f^{t+}(\bar x)$ and $f^{t-}(\bar x)$ are parallel.
In this case, $\pi_{ijk} $ is crossed by Krasovsky solutions, unless $f^{t+}(\bar x)$ is tangent to the manifold, but this may occur only on a set of codimension at least two.
Let us then consider the 
case 
$\nabla\theta_{ijk} (\bar x)\cdot f^{t+}(\bar x)-
\nabla \theta_{ijk} (\bar x)\cdot f^{t-}(\bar x)> 0$
and analyze different possibilities. 
\begin{itemize}
    \item Case $\nabla\theta_{ijk} (\bar x)\cdot f^{t+}(\bar x)>0$. If also $\nabla\theta_{ijk} (\bar x)\cdot f^{t-}(\bar x)>0$, then Krasovsky solutions
cross $\pi_{ijk}$. If 
$\nabla\theta_{ijk} (\bar x)\cdot f^{t-}(\bar x)\leq 0$, then Krasovsky solutions can either leave $\pi_{ijk}$ or slide on it. 
\item Case $\nabla\theta_{ijk} (\bar x)\cdot f^{t+}(\bar x)\leq 0$. Then  $\nabla\theta_{ijk} (\bar x)\cdot f^{t-}(\bar x)<0$ and Krasovsky solutions  cross $\pi_{ijk}$.
\end{itemize}

We conclude that solutions, which originate from outside $\pi_{ijk}$ and reach it, must cross it. Therefore, uniqueness can only fail for (sliding) solutions that originate inside $\pi_{ijk}$.
After this informal argument for $\kappa=1$, we proceed to give a complete proof for any $\kappa$, thereby completing the proof of  Theorem \ref{t-exun}. 

\begin{prop}[Uniqueness from almost any initial datum, topological]
\label{th:aeuniq-t}
The set of initial data from which there exist more than one Krasovsky solutions for \eqref{eq:tbc} has zero Lebesgue measure in $\R^{nN}$.
\end{prop}
\begin{proof} 
Fix an initial condition $\bar{x}$ and let $X_{\bar{x}}$ 
be the set of solutions  $x(\cdot)$ to \eqref{eq:tbc} such
that  $x(0)=\bar{x}$ defined on $[0,T(x(\cdot))[$, with $0<T(x(\cdot))\leq +\infty$. Define
\begin{equation}\label{eq:def-tU}
t_U=\inf \{t:  \exists x(\cdot),y(\cdot)\in X_{\bar{x}}, 
t\leq\min \{T(x(\cdot)),T(y(\cdot))\}, x(t)\not= y(t)\},
\end{equation}
and
\begin{equation}\label{eq:def-A}
\mathcal{A}=\{\bar{x}\in \R^{nN}\setminus\MM: t_U<+\infty\}.
\end{equation}
Notice that $\MM$ is a stratified set of codimension 1, thus of zero Lebesgue measure in  $\R^{nN}$. Therefore the statement
is equivalent to prove that $\mathcal{A}$ has zero Lebesgue measure.
For $\bar x\in \mathcal{A}$, we define:
\begin{equation}
\tilde{t}=\inf \{t: \exists x(\cdot)\in X_{\bar{x}}, x(t)\in\MM\}.
\end{equation}
Since \eqref{eq:tbc} is Lipschitz continuous on $ \R^{nN}\setminus\MM$,
there exists a unique solution in $X_{\bar{x}}$ at least until reaching $\MM$, thus $\tilde{x}=x(\tilde{t})\in\MM$ depends only on $\bar{x}$.
Now define the set of indexes
\[
J=\{(i,j,k,i',j',k'): (i,j,k)\not= (i',j',k')\},
\]
and the stratified sets:
\[
\MM_{ijki'j'k'}=
\MM_{ijk}\cap \MM_{i'j'k'},
\]
We now analyze the dynamics on
$\MM\setminus \left(\cup_{(i,j,k,i',j',k')\in J}
{\MM}_{ijki'j'k'}\right)$ 
to identify a stratified set of codimension
two out of which trajectories cross
$\MM$ transversally.
Consider now $x\in\MM_{ijk}$, 
assume  $(i,j,k)$ is the unique index
for which $x\in\MM_{ijk}$. 
We also assume
$\|x_i-x_j\|=\max_{\ell\in N^t_i(x)}\|x_i-x_\ell\|$, i.e. $j$ is among the farthest $\kappa\geq 1$ neighbours of $i$, otherwise $N^t_i$ is constant in a neighbor
of $x$ and uniqueness holds.
Since $(i,j,k)$ is the unique index
for which $x\in\MM_{ijk}$, we indeed have that
$\max_{\ell\in N^t_i(x)}\|x_i-x_\ell\|$
is achieved exactly for indexes $j$ and $k$.
Now, set $P_i= N^t_i(x)\setminus \{j,k\}$,
$P_j= N^t_j(x)$,
$P_k= N^t_k(x)$.
Define the following:
\begin{equation}
f_m(x)=\sum_{\ell\in P_m} a(\|x_\ell-x_m\|)(x_\ell-x_m)
\end{equation}
for $m=i,j,k$.
Then a Krasovsky solution $y(\cdot)$ with $y(0)=x$, if differentiable at $0$, satisfies:
\[
\dot{y}_i(0)=f_i(x)+\alpha\,
a(\|x_j-x_i\|)(x_j-x_i)+(1-\alpha)\,
a(\|x_k-x_i\|)(x_k-x_i),
\]
for some $\alpha\in [0,1]$, and:
\[
\dot{y}_j(0)=f_j(x)
,\quad
\dot{y}_k(0)=f_k(x).
\]
Recall the definition of the function
$\theta_{ijk}$ given in \eqref{eq:def-thetaijk}.
If $\theta_{ijk}$ computed along $y(\cdot)$ is 
differentiable at $0$, then:
\begin{equation}\label{eq:der-thetaijk}
\dot{\theta}_{ijk} (0)=
C(x)
+2\alpha\, a(\|x_i-x_j\|)\,(x_j-x_i)\cdot (x_k-x_j)
+2(1-\alpha)\, a(\|x_i-x_j\|)\,(x_k-x_i)\cdot (x_k-x_j)
\end{equation}
where we used $\|x_i-x_j\|^2=\|x_i-x_k\|^2$
and
\begin{equation}\label{eq:C}
C(x)= 2    (f_i-f_j)\cdot (x_i-x_j) - 2 (f_i-f_k)\cdot (x_i-x_k).
\end{equation}
Define the stratified sets 
\begin{eqnarray*}
\widehat{\MM}_{ijk}&=&\{
x\in\MM_{ijk} : 
C(x)
+2\, a(\|x_i-x_j\|)\,(x_j-x_i)\cdot (x_k-x_j)=0,\\
&&\mbox{~~or~~} C(x)+2\, a(\|x_i-x_j\|)\,(x_k-x_i)\cdot (x_k-x_j)= 0
\},
\end{eqnarray*}
and finally
\[
\widehat{\MM}=
\left(\cup_{ijk} \widehat{\MM}_{ijk}\right)\bigcup
\left(\cup_{(i,j,k,i',j',k')\in J}
{\MM}_{ijki'j'k'}\right).
\]
Notice that $\widehat{\MM}$ is of codimension 2, and we state the following claim:\\
{\bf Claim a)} If $\tilde{x}\in \MM\setminus\widehat{\MM}$, then
there exists  $\epsilon>0$ such that  $x(t)\notin\MM$
for $t\in ]\tilde{t},\tilde{t}+\epsilon[$, and
$x\equiv y$ on $[0,\tilde{t}+\epsilon[$
for every $x(\cdot)$, $y(\cdot) \in X_{\bar{x}}$.\\
To prove Claim a), let $(i,j,k)\in I$ be the unique triplet
such that $\tilde{x}\in\MM_{ijk}$.
Assume $j\in N^t_i(\tilde x)$ or $k\in N^t_i(\tilde x)$,
otherwise the claim is obvious.
The function $\theta_{ijk}$ computed along $x(\cdot)$
satisfies
$\theta_{ijk}(\tilde{t})=0$,
is twice continuously differentiable on $[0,\tilde{t}[$ with bounded derivatives, thus we can define
$\tilde{\xi}:=\lim_{t\to\tilde{t}-}\dot{\theta}_{ijk}(t)$.\\
First assume $\tilde{\xi}>0$: then, there exists
$\delta >0$ such that both 
$\theta_{ijk}(t)<0$  and $j\in N^t_i(x(t))$ on 
$]\tilde{t}-\delta,\tilde{t}[$. 
Then, possibly restricting $\delta>0$, on  $]\tilde{t}-\delta,\tilde{t}[$ we have
$\dot{\theta}_{ijk}(t)=
C(x(t)) +2\, a(\|x_i(t)-x_j(t)\|)\,(x_j(t)-x_i(t)) \cdot (x_k(t)-x_j(t))>0$
and
$\tilde{\xi}=C(\tilde{x}) +2\, a(\|\tilde{x}_i-\tilde{x}_j\|)\,(\tilde{x}_j-
\tilde{x}_i)
\cdot (\tilde{x}_k-\tilde{x}_j)>0$.
Since $x(\cdot)$ is a Krasovsky solution,
for almost every time $\dot{\theta}_{ijk}(t)$
can be computed as in \eqref{eq:der-thetaijk}
for some $\alpha(t)\in [0,1]$.
From $(x_k-x_i)=(x_k-x_j)+(x_j-x_i)$,
we get $(x_k-x_i)\cdot (x_k-x_j)
= (x_j-x_i)\cdot (x_k-x_j)+\|x_k-x_j\|^2$,
thus $\dot{\theta}_{ijk}(t)>0$
for $t$ sufficiently close to $\tilde{t}$.
This implies that there
exists $\epsilon>0$ such that
$\theta_{ijk}>0$, $j\notin N^t_i(x(t))$
and  $k\in N^t_i(x(t))$
for $t\in ]\tilde{t},\tilde{t}+\epsilon[$.
Since $\tilde{x}\notin \widehat{\MM}$, 
by possibly reducing $\epsilon$, it holds
$x(t)\notin \widehat{\MM}$.
In particular, all sets $N^t_\ell(x(t))$ are constant
for $t\in ]\tilde{t},\tilde{t}+\epsilon[$.
Thus we conclude that Claim a) holds.\\
The case $\tilde{\xi}<0$ can be treated similarly,
while the case  $\tilde{\xi}=0$ is excluded
since $\tilde{x}\notin \widehat{\MM}_{ijk}$.

Now set:
\begin{equation}
t_{\widehat{\MM}}=
\inf \{t: \exists x(\cdot)\in X_{\bar{x}}, x(t)\in\widehat{\MM}\},
\end{equation}
Claim a) ensures  $t_{\widehat{\MM}}\leq t_U$, as uniqueness can be lost only when crossing $\widehat{\MM}$.
Therefore, if $\bar{x}\in \mathcal{A}$, then every $x(\cdot)\in X_{\bar{x}}$ is Lipschitz continuous and  coincides (at least) up to $t_{\widehat{\MM}}$. 
This implies that 
$H^{1+\epsilon}(\{x(t):t\in [0,t_{\widehat{\MM}}],x(\cdot)\in X_{\bar{x}}\})=0$ for every $\epsilon>0$, where
$H^r$ is the Hausdorff measure of dimension $r$ in $\R^{nN}$. Since $\widehat{\MM}$ is of codimension 2,
by Fubini Theorem, for $0<\epsilon<1$ we have:
\[
H^{nN}(\mathcal{A})\leq \int_{\widehat{\MM}} 
\left(H^{1+\epsilon}(\{x(t):t\in [0,t_{\widehat{\MM}}],x(\cdot)\in X_{\bar{x}}\})
\right)
d H^{nN-2+\epsilon}(\bar{x})=0.
\]
The measure $H^{nN}$ coincides with the Lebesgue measure on $\R^{nN}$, thus  $\mathcal{A}$ has zero Lebesgue measure.
\end{proof}

\section{Asymptotic behavior of solutions}

We now study convergence to cluster points, i.e.\ Property P3). We first need to investigate the relationships between cluster points and equilibria for Caratheodory and Krasovsky solutions to \eqref{eq:mbc} and \eqref{eq:tbc}.

\subsection{Equilibria and cluster points}

We begin by recalling that cluster points are points
$x^{\infty}=(x^\infty_1,\ldots,x^\infty_N)$, $x^\infty_i\in\R^n$, such that for every $i\in V$,  for every $j\in N_i(x^{\infty})$ it holds $x^\infty_i=x^\infty_j$.
It is easy to prove that cluster points are Caratheodory equilibria for both \eqref{eq:mbc} and \eqref{eq:tbc}. One can ask whether all equilibria are cluster points: here, the metric and topological models are completely different. 
For the metric bounded confidence model, all Krasovsky equilibria are indeed cluster points (see \cite[Prop.~7.1]{piccoli_rossi_2020}). Instead, for the topological bounded confidence model, there exist equilibria that are not cluster points, in the following cases:
\begin{itemize}
    \item for $\kappa>1$, for any kind of solutions, see Example \ref{ex:C-eq-no-cluster};
    \item for Krasovsky solutions even with $\kappa=1$, see Example \ref{ex:K-eq-no-cluster-k=1}.
\end{itemize}
Instead, Caratheodory equilibria for $\kappa=1$ are all clusters, as proved in Proposition~\ref{p-eq-top-k=1}.

\begin{example}[Non-convergence to clusters, topological $\kappa \ge2$]
\label{ex:C-eq-no-cluster}
Let $N=7, \  n=1, \  \kappa=2$, $a\equiv 1$  and consider the point $\overline x$ with
$\overline  x_2=\overline x_4=\overline x_5=0$, 
$\overline x_1=\frac{1}{2}$, 
$\overline x_{3}=\overline x_{6}=\overline x_{7}=1$.
It can be easily computed that $f^t(\overline x)=0$ and therefore $\overline x$ is an equilibrium point with respect to classical, Caratheodory and Krasovsky solution. Moreover, it is not a cluster point as there is just   $1<\kappa=2$ index with value $\frac{1}{2}$. We  remark that $\overline x$ is {\em not} locally attractive with respect to all types of solutions.
\end{example}
As equilibria correspond to constant solutions, this example also shows that solutions of the topological bounded confidence model do not satisfy property P3), in general.

In the previous example we fixed $\kappa=2$. For $\kappa=1$, we will show that all classical and Caratheodory equilibria are cluster points, as stated by the following proposition.

\begin{prop}[Caratheodory equilibria, topological $\kappa=1$]\label{p-eq-top-k=1}
If $\kappa=1$, Caratheodory equilibria of \eqref{eq:tbc} are cluster points.  
\end{prop}

\begin{proof}
As $\kappa=1$, cluster points are such that agents can be divided in groups of at least two agents with the same value. The $i$-th component of the vector field writes  $f^t_i(x)=a(\| x_{\Gamma (i)}- x_i\|) (x_{\Gamma (i)}- x_i)$, where  $\Gamma (i)$  is the (state dependent) neighbor of $i$.  Note that $f^t_i(x)$ is null if and only if $x_{\Gamma (i)}= x_i$, as $a(r)>0$ for $r>0$. Then,  $\overline x_i=\overline x_{\Gamma (i)} $ for all $i\in V$ and $\overline x$ is a cluster point.
\end{proof}

Krasovsky equilibria which are not cluster points appear even if $\kappa=1$, as shown by the following example. Their existence implies that, even if $\kappa=1$, Krasovsky solutions do not necessarily converge to cluster points.

\begin{example}[Non-convergence of Krasovsky to clusters, topological $k=1$]\label{ex:K-eq-no-cluster-k=1}
We consider \eqref{eq:tbc} with $N=5$, $n=1$, $\kappa=1$, $a\equiv 1$  and the initial condition
$\overline x=(-1,1,0,1,-1)$. Observe that
$\overline x$ is a discontinuity point of the vector field $f^t$. Among the limit values of the vector field $f^t$ at $\overline x$ there are $(0,0,-1,0,0)$ and $(0,0,1,0,0)$.
Then $0\in Kf^t(\overline x) $ and $\overline x$ is a Krasovsky equilibrium which is not a cluster point. 
\end{example}

\subsection{P3) Convergence to cluster points}

The following proposition summarizes our results about convergence to cluster points. The result is the best possible one in terms of convergence to clusters, given the above counterexamples.

\begin{prop}[Convergence to cluster points] \label{p-Fclust}
\begin{itemize}
    \item[(i)] {\bf Metric bounded confidence.} For any  Krasovsky solution of \eqref{eq:mbc}, Property P3)  holds.
    \item[(ii)] {\bf Topological bounded confidence.} If $\kappa=1$, for any  Caratheodory solution of \eqref{eq:mbc}, Property P3)  holds.
    \end{itemize}
\end{prop}

The proof of (i) is given in \cite[Prop.~7.1]{piccoli_rossi_2020} and = is based on the observation that \eqref{eq:mbc} can be written as a gradient flow as follows. Define 
$$\Phi_{ij}(r)=\begin{cases}
\int_0^r a(s)s\,ds&\mbox{~~for~}r<1\\
\int_0^1 a(s)s\,ds&\mbox{~~for~}r\geq 1
\end{cases}
$$
and observe that, if $\|x_i-x_j\|\neq 1$, for every $i\not= j$, then
$$\dot x_i=-\sum_{j\neq i} \nabla \Phi_{ij}(\|x_i-x_j\|).$$
This suggests to define the candidate Lyapunov function
$$V(x)=\sum_{i,j\neq i} \Phi_{ij}(\|x_i-x_j\|),$$ 
which satisfies $\dot V(x(t))\leq 0$ for a.e.\ time and allows (despite being nonsmooth and non-proper) to establish an ad-hoc convergence argument.

The proof of (ii) requires a slightly different reasoning. The special case of (piecewise) classical solutions in dimension $n=1$ was proved in \cite{ceragioli2020modeling} by exploiting the special structure of its induced graph $G(x)$, which we describe next. 

\begin{prop}[Directed pseudoforest]\label{p-topology-case-k=1}
If $\kappa=1$, then for all $n$, for  every $x\in\real^{nN}$, the interaction graph $G(x)$ of \eqref{eq:tbc} is the union of weakly connected components, such that each component contains exactly one circuit of length 2 and the two nodes of the circuit can be reached from all nodes of the component.
\end{prop}

\begin{proof}  
Let $ x$ be fixed and consider a connected component of $G(x)$, called $G'$. We first prove that $G'$ has exactly one circuit of length $2$. 
Let $M$ be the number of nodes of the connected component. As any node has exactly one out-edge, the number of edges of the component is exactly $M$, then $G'$  contains one circuit (this kind of graph is referred to as directed pseudoforest). Furthermore, we can observe that the nodes of the circuit are reachable from any node in $G'$. As any node has an outgoing edge, starting from any node there exists an infinite  walk. As the number of nodes is finite, it must contain  a circuit. This means that the walk contains  the nodes of the circuit. 

We now prove that any circuit cannot have length greater than $2$. Assume by contradiction that there is a circuit with length $p>2$. Let $i_1,..., i_p$ be its nodes and $i_1$ be the smallest index.
Thanks to the definition of neighbour, it must hold
$$
\|x_{i_1} -x_{i_2}\| \leq \|x_{i_1} -x_{i_2}\|\leq ...\leq \|x_{i_1} -x_{i_p}\|.
$$
If 
$
\|x_{i_1} -x_{i_2}\| < \|x_{i_1}   -x_{i_p}\|
$ then $i_p$ is the neighbour of $i_2$ instead of $i_1$,   contradiction.
Then it must hold
$\|x_{i_1} -x_{i_2}\| = \|x_{i_1} -x_{i_2}\|= ...= \|x_{i_1} -x_{i_p}\|
$. In this case 
$i_1$ should be the neighbour of all nodes $i_3, i_4,..., i_{p-1}$, as it is the smallest index.
Finally $G'$ has exactly one circuit: indeed, to connect two  circuits there should be a node with out-degree at least $2$.
\end{proof}

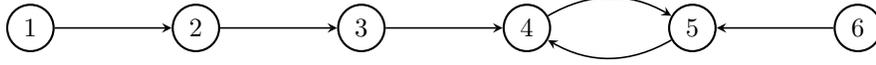
\begin{figure}
\begin{center}
    \begin{tikzpicture}[
            > = stealth, 
            shorten > = 0pt, 
            auto,
            node distance = 22mm, 
            semithick 
        ]

        \tikzstyle{every state}=[
            draw = black,
            thick,
            fill = white,
            minimum size = 4mm
        ]

        \node[state] (1) {$1$};
        \node[state] [right of=1] (2) {$2$};
        \node[state] [right of=2] (3) {$3$};                
        \node[state] [right of=3] (4) {$4$};
        \node[state] [right of=4] (5) {$5$};        
        \node[state] [right of=5] (6) {$6$};                
        
        \path[->] (1) edge (2);
        \path[->] (2) edge (3);
        \path[->] (3) edge (4);
        \path[->, bend left] (4) edge (5);   
        \path[->, bend left] (5) edge (4);
        \path[->] (6) edge (5);
        
        \end{tikzpicture}
        \end{center}
\caption{Example of weakly connected component of graph $G(\overline x)$, where $N=6$, $n=1$, $\kappa=1$,  $\overline x=(0, 10, 19, 27, 28, 30) $.}\label{fig:weak-graph}
\end{figure}

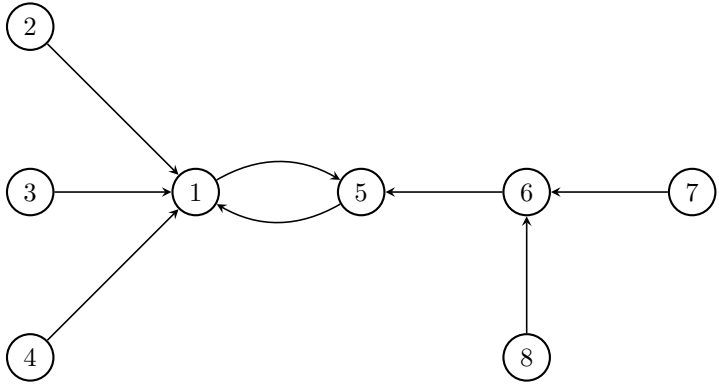
\begin{figure}
\begin{center}
    \begin{tikzpicture}[
            > = stealth, 
            shorten > = 0pt, 
            auto,
            node distance = 22mm, 
            semithick 
        ]

        \tikzstyle{every state}=[
            draw = black,
            thick,
            fill = white,
            minimum size = 4mm
        ]

        \node[state] (1) {$2$};
        \node[state] [below of=1] (2) {$3$};
        \node[state] [below of=2] (3) {$4$};                
        \node[state] [right of=2] (4) {$1$};
        \node[state] [right of=4] (5) {$5$};        
        \node[state] [right of=5] (6) {$6$};           \node[state] [right of=6] (7) {$7$};
        \node[state] [below of=6] (8) {$8$};
        
        \path[->] (1) edge (4);
        \path[->] (2) edge (4);
        \path[->] (3) edge (4);
        \path[->, bend left] (4) edge (5);   
        \path[->, bend left] (5) edge (4);
        \path[->] (6) edge (5);
        \path[->] (7) edge (6);
        \path[->] (8) edge (6);
        
        \end{tikzpicture}
        \end{center}
\caption{Example of weakly connected component of graph $G(\overline x)$ where $N=8$, $n=2$, $\kappa=1$,  $\overline x=((0,0) (0,1),(-1,0), (0,-1), (1/2,0),(1,0),(1,1),(1,-1)) $.}\label{fig:weak-graph-2}
\end{figure}

An interaction graph with the structure of $G(x)$, if kept static, would guarantee convergence to consensus for each connected component and, therefore, convergence to a cluster point. However, the graph $G(x(t))$ evolves with time in such a way that connected components can split and distinct connected components can merge. The latter phenomenon is illustrated in the following example.

\begin{example}[Merging components in Caratheodory solutions]
Let $N=4, \, n=1, \,  \kappa=1,\ a\equiv 1$ and consider the initial condition  $\overline{x}=(-1,0,1,1)$. Consider the Caratheodory solution $x(t)=(1-t e^{-t} -2e^{-t},1-e^{-t},1,1)$. Note that $x(0)=\overline x $ and $x(t)$ that satisfies \eqref{eq:tbc} at all $t>0$ but not at $t=0$. The graph $G(x(0))$ has two connected components whose vertices are $\{1,2\}$ and $\{3,4\}$ whereas $G(x(t))$ is connected for all $t>0$.  
\end{example}

This counterexample prevents leveraging the topology of $G(x)$ to prove (ii) for Caratheodory solutions. We therefore resort to a Lyapunov-like argument, which is partly inspired by the one  in \cite{piccoli_rossi_2020} for metric interactions, but will be valid for topological interactions in the case of $\kappa=1$ only.

We introduce the integral function
$$I(r):=\int_0^r a(s)s\,ds$$ 
and write the candidate Lyapunov function
\begin{equation}\label{e-W}
W(x):=\sum_{i,j\in N_i^t(x)} I(\|x_j-x_i\|).
\end{equation}
One might hope to write \eqref{eq:tbc} as 
    $\dot x=-\nabla W (x)$,
like in the metric case. This is false, as one can easily observe  that this expression entails interactions  that are symmetric, while this is not the case for \eqref{eq:tbc}.



 

We will anyway be able to prove that $W(x)$ is a Lyapunov function for solutions to \eqref{eq:tbc}. This is only the case for $\kappa=1$ and for Caratheodory solutions, but the proof is quite different from the case of metric bounded-confidence \eqref{eq:mbc}, again due to the asymmetry of the interactions. Instead, $W(x)$ is not a Lyapunov function, neither for Caratheodory solutions with $\kappa>1$ nor for Krasovsky solutions with $\kappa\geq 1$, as shown by Examples \ref{ex:noLyap-1} and \ref{ex:noLyap-2} below.

\begin{prop}[$W$ is Lyapunov] Let $\kappa=1$. Then, the function $W(x(t))$ is continuous and non-increasing for Caratheodory solutions.
\end{prop}
\begin{proof}
The proof is based on rewriting $W(x)=\sum_{i=1}^N W_i(x)$ where 
\begin{equation}\label{eq:W_i}
W_i(x)=\min_{j\neq i} I(\|x_i-x_j\|).     
\end{equation}
It is clear that both $I(r)$ and $x(t)$ are continuous. Then, both all  $W_i(x(t))$ and their sum $W(x(t))$ are continuous too.
The rest of the proof is based on Danskin theorem\footnote{In Danskin notation, we have $F=F(x,j)=I(\|x_i-x_j\|)$ maximized with respect to  $j\in V\setminus\{i\}$. } \cite{danskin} for $W_i(x)$. Similarly to the proof of Proposition~\ref{p-contractive}, even though  $W_i(x)$ can be non-differentiable, it admits directional derivative with respect to any direction. We apply it to our function, denoting the direction of displacement with $h=(h_1,\ldots,h_N)$, where each $h_k$ is the $n$-dimensional direction of displacement of the position of the agent $k$. By applying Danskin formula, the directional derivative $D_h$ along $h$ is given by
\begin{eqnarray*}
D_h W_i(x)&=&\min_{j\in A_i(x)}
\sum_{k=1}^N h_k \cdot \nabla_{x_k} I(\|x_i-x_j\|)=\min_{j\in A_i(x)}
\sum_{k\in\{i,j\}} h_k \cdot \nabla_{x_k} I(\|x_i-x_j\|)\\
&=&\min_{j\in A_i(x)}\left(h_i \cdot a(\|x_i-x_j\|)(x_i-x_j)- h_j \cdot a(\|x_i-x_k\|)(x_i-x_j)\right)\\
&=&\min_{j\in A_i(x)}a(\|x_i-x_j\|)(h_i-h_j)\cdot(x_i-x_j),
\end{eqnarray*}
where $A_i(x)$ is the set of indexes $j\neq i$ realizing $\min I(\|x_i-x_j\|)$.

We now apply this formula to compute the time derivative $\dot W_i(x(t))$, whenever it exists. Following the same computations as for \eqref{e-derdir}, we have $\dot W_i(x(t))=D_{\dot x(t)}W_i(x(t))$. 
Since the time derivative $\dot x(t)$ exists for almost every time $t\in(0,T)$, this holds for $\dot W_i(x(t))$ too. We compute this derivative, by restricting ourselves to Caratheodory solutions, that satisfy $\dot x_i=a(\|x_i-x_k\|)(x_k-x_i)$ for almost every time, with $k\in N_i^t(x)$. Denote with $L:=\|x_i-x_k\|$ and with $l$ the unique element $l\in N_k^t(x)$. Observe that $l\in N_k^t(x)$ implies $\|x_l-x_k\|\leq \|x_i-x_k\|=L$, that in turn implies $a(\|x_l-x_k\|)\leq a(\|x_i-x_k\|)=a(L)$ and $(x_l-x_k)\cdot(x_i-x_k)\geq -L^2$.  
Since $k\in A_i(x)$, and using previous estimates, it holds
\begin{eqnarray}
D_{\dot x(t)}W_i(x(t))&=&\min_{j\in A_i(x)}a(\|x_i-x_j\|)(\dot x_i-\dot x_j)\cdot(x_i-x_j)\leq a(L)(\dot x_i-\dot x_k)\cdot(x_i-x_k)=\label{e-DWi}\\
&&a(L)\left(a(\|x_i-x_k\|)(x_k-x_i)\cdot (x_i-x_k)- a(\|x_k-x_l\|)(x_l-x_k)\cdot(x_i-x_k)\right)\leq\nonumber\\
&& -a(L)^2L^2+a(L)^2L^2=0.\nonumber
\end{eqnarray}
Since $W_i(x(t))$ is continuous, this implies that each $W_i(x(t))$ is non-increasing.

Passing to $W(x(t))$, that is a finite sum of continuous and non-increasing functions, the proof follows.
\end{proof}

\begin{example}[$W(x(t))$ increasing if $\kappa>1$]\label{ex:noLyap-1} We now prove that $W(x)$ given in \eqref{e-W} is not a Lyapunov function for \eqref{eq:tbc} in the case $\kappa>1$ for Caratheodory solutions (hence for Krasovsky solutions too). Consider the following initial configuration of $N=8$ agents on the real line with $\kappa=2$ and $a(r)\equiv 1$: choose $\bar x=(-9,-9,-9,-2,2,9,9,9)$ and observe that the unique solution of \eqref{eq:tbc} in the Krasovsky sense (that is even classical and Caratheodory) is given by
\begin{equation}
    \begin{cases}
         \dot x_1=\dot x_2=\dot x_3=\dot x_6=\dot x_7=\dot x_8=0\\
         \dot x_4=(x_5-x_4)+(x_1-x_4)\\
         \dot x_5=(x_4-x_5)+(x_6-x_5).
    \end{cases}
\end{equation}
By symmetries, it holds $x_4(t)=-x_5(t)$, thus $\dot x_5=9-3 x_5$. It then holds $x_5(t)=3-e^{-3t}$.
Notice that the topology does not change and the  solution converges to the equilibrium point $x^{\infty}=(-9,-9,-9,-3,3,9,9,9)$ which is not a cluster point. A direct computation gives
\begin{equation}
W(x(t))=\tfrac12 \left((x_1(t)-x_4(t))^2+2 (x_4(t)-x_5(t))^2+(x_5(t)-x_6(t))^2\right)=4x_5^2(t)+(9-x_5(t))^2,
\end{equation}
that satisfies $\dot W(x(t))=(10x_5(t)-18)\dot x_5(t)>0$ for all $t\in[0,+\infty)$.
\end{example}

\begin{example}[$W(x(t))$ increasing for Krasovsky solutions]\label{ex:noLyap-2} We now prove that $W(x)$ given in \eqref{e-W} is not a Lyapunov function for Krasovsky solutions with $\kappa=1$. Consider the following initial configuration of $N=5$ agents on the real line and $a(r)\equiv 1$: choose $\bar x=(-1-y_0,-1+y_0,0,1-y_0,1+y_0)$  with $y_0\in\left(0,\frac1{17}\right)$. Observe that one of the solutions of \eqref{eq:tbc} in the Krasovsky sense is 
\begin{equation*}
    \begin{cases}
         x_1(t)=-1-y(t),\qquad
         x_2(t)=-1+y(t),\qquad
         x_3(t)=0,\qquad 
         x_4(t)=1-y(t),\qquad
         x_5(t)=1+y(t),    \end{cases}
\end{equation*}
with $y(t)=\exp(-2t)y_0$. Indeed, for all $t>0$ this solution satisfies
\begin{eqnarray*}\dot x&=&\frac12 (x_2-x_1,x_1-x_2,x_4,x_5-x_4,x_4-x_5)+\frac12 (x_2-x_1,x_1-x_2,x_2,x_5-x_4,x_4-x_5)\\
&=&\left(2y(t),-2y(t),0,2y(t),-2y(t)\right).
\end{eqnarray*}
A direct computation gives
\begin{eqnarray*}
2W(x(t))&=&2(x_1(t)-x_2(t))^2+(x_3(t)-x_2(t))^2+
2(x_4(t)-x_5(t))^2=\\
&&4(2y(t))^2+(1-y(t))^2=1-2y(t)+17y(t)^2.
\end{eqnarray*} 
Its derivative is $4y(t)-17\cdot 4 y(t)^2$, that is positive for $y(t)\in\left(0,\frac1{17}\right)$. This holds whenever $y_0\in\left(0,\frac1{17}\right)$. As a consequence, $W(x(t))$ is strictly increasing.
\end{example}

We are now ready to describe the structure of the limits of Caratheodory solutions of \eqref{eq:tbc} with $\kappa=1$, that are indeed clusters.

\begin{prop}[Convergence and cluster properties, topological $\kappa=1$]\label{p-Tclust} Let $x_1(t),\ldots,x_N(t)$ be a Caratheodory solution of \eqref{eq:tbc} with $\kappa=1$. Then, the following clustering properties hold:
\begin{itemize}
\item each agent satisfies $\lim_{t\to+\infty} x_i(t)=x_i^\infty$ for some $x_i^\infty \in \R^n$;
\item for each $i\in V$ there exists at least one $j\neq i$ such that $x_i^\infty=x_j^\infty$.
\end{itemize}
 This also implies that P3) holds and  $W(x^\infty)=0$.
\end{prop} 
\begin{proof} First recall that $x(t)$ is bounded, due to contractivity of the support proved in Proposition~\ref{p-contractive}. This implies that $a(\|x_i(t)-x_j(t)\|)$ is bounded too, as $a(r)$ is Lipschitz continuous by hypothesis. This in turn implies that both $x(t)$ and $a(\|x_i(t)-x_j(t)\|)$ are Lipschitz continuous too. Boundedness also implies that the $\omega$-limit is bounded.

Fix now any $x^*=(x_1,\ldots,x_N)$ in the $\omega$-limit of $x(t)$. By definition, there exists a sequence $t_k\to+\infty$ such that $x(t_k)\to x^*$. Fix $\eps>0$ and $K=K_\eps$ sufficiently large to have
\begin{eqnarray}
\|((x_i(t_k)-x_j(t_k))\cdot (x_l(t_k)-x_m(t_k)))-(x_i^*-x_j^*)\cdot (x_l^*-x_m^*)\|<2\eps
\end{eqnarray}
for all $i,j,l,m\in V$ and $k>K_\eps$. Since trajectories are bounded and  Lipschitz continuous, there exists a uniform $\delta>0$ such that 
\begin{eqnarray} \label{e-eps3}
\|((x_i(t_k+\tau)-x_j(t_k+\tau))\cdot (x_l(t_k+\tau)-x_m(t_k+\tau)))-(x_i^*-x_j^*)\cdot (x_l^*-x_m^*)\|<\eps
\end{eqnarray}
for all $\tau\in(-N^{3N}\delta,N^{3N}\delta)$.

Fix now $i=1$, recall \eqref{eq:W_i}
and consider the derivative 
\begin{equation}\label{e-Dbase}
    D_{\dot x(t)}W_1(x(t))=\min_{j\in A_1(x(t))}a(\|x_1(t)-x_j(t)\|)(\dot x_1(t)-\dot x_j(t))\cdot(x_1(t)-x_j(t)),
\end{equation}
whenever $\dot x(t)$ is well-defined, i.e.\ for almost every $t>0$. Since the number of nearest neighbours  of $x_1$ in $A_1(x)$ is $N-1$ at most, there exists at least one index ${j_1^k}$ such that ${j_1^k}$ is the minimizer in the right hand side of \eqref{e-Dbase} for all $\tau\in I^{1,a}_k$, where $I^{1,a}_k\subset (t_k-N^{3N}\delta, t_k+N^{3N}\delta)$ has Lebesgue measure $2N^{3N-1}\delta$. Since the number of possible ${j_1^k}$ is finite, eventually passing to a subsequence in $k$, we assume that ${j_1^k}={j_1}$ is constant. By recalling that Caratheodory solutions satisfy the  dynamics  \eqref{eq:tbc} for almost every time, it holds 
\begin{eqnarray*}\label{e-D}
    D_{\dot x(t)}W_1(x(t))&=&a(\|x_1(t)-x_{j_1}(t)\|)[a(\|x_1(t)-x_l(t)\|)(x_l(t)-x_1(t))\\
    &&~~~~~~~~~ -a(\|x_{j_1}(t)-x_m(t)\|)(x_m(t)-x_{j_1}(t))]\cdot(x_1(t)-x_{j_1}(t)),
\end{eqnarray*}
for almost every $t\in I^{1,a}_k$, where $l\in N_1^t(x(t))$ and $m\in N_{j_1}^t(x(t))$. Again, since the number of possible neighbours of $1$ in $N_1^t(x(t))$ is $N-1$ at most, then there exists $l^k_1$ such that $l^k_1\in N_1^t(x(t))$ for all $t\in I^{1,b}_k$ where $I^{1,b}_k\subset I^{1,a}_k$ has Lebesgue measure $2N^{3N-2}\delta$. By passing to a subsequence in $k$, we can assume $l_1$ constant. With a similar argument, we can find $m_1$ and $I^{1,c}_k\subset I^{1,b}_k$ with Lebesgue measure $2N^{3N-3}\delta$ such that $m_1\in N_{j_1}^t(x(t))$ for all $t\in I^{1,c}_k$.

We now  choose the index 2 and define the corresponding indexes $j_2,l_2,m_2$ and sets $I^{2,c}\subset I^{2,b}\subset I^{2,a}_k\subset I^{1,c}_k$, each with Lebesgue measure being $1/N$ of the previous one. We then move to indexes $3,4,\ldots N$, finally reaching $I_k:=I^{N,c}_k$ with Lebesgue measure $2\delta$ and such that, for each $i\in V$ there exists corresponding $j_i,l_i,m_i$ such that for all $\tau\in I_k$ the following hold:
\begin{itemize}
    \item the index $j_i$ is the minimizer in the right hand side of \eqref{e-Dbase};
    \item the index $l_i$ is the unique element of $N_i^t(x(\tau))$;
    \item the index $m_i$ is the unique element of $N_{j_i}^t(x(\tau))$.
\end{itemize} 

Fix now any $i\in V$ and the corresponding $j_i,l_i,m_i$ defined above. We now prove that it holds
\begin{eqnarray}\label{e-AA}
\mathcal{A}_i:=a(\|x_i^*-x_{j_i}^*\|)\left[a(\|x_i^*-x_{l_i}^*\|)(x_{l_i}^*-x_i^*)\cdot(x_i^*-x_{j_i}^*)-a(\|x_{{j_i}}^*-x_{m_i}^*\|)(x_{m_i}^*-x_{j_i}^*)\cdot(x_i^*-x_{j_i}^*)\right]=0.
\end{eqnarray}

By contradiction, first assume that $\mathcal{A}_i>0$: then, observe that \eqref{e-D} coupled with \eqref{e-eps3}, implies that there exists $\bar k$ such that $D_{\dot x(\tau)}W_i(x(\tau))> \mathcal{A}_i/2$ for every $\tau \in I_k$ with $k\geq \bar k$. Since the set of such $\tau$ has non-zero Lebesgue measure, this contradicts the fact that $W_i(x(t))$ is a non-increasing function.

Assume now $\mathcal{A}_i<0$ and use the same reasoning to prove that $D_{\dot x(\tau)}W_i(x(\tau))< -|\mathcal{A}_i|/2$ for every $\tau\in I_k$ with $k\geq \bar k$. Since for all times in $(t_k-N^{3N}\delta,t_k+N^{3N}\delta)$ we have $W_i(x(t))$ non-increasing, we can write
\begin{eqnarray}
W_i(x(t_k+N^{3N}\delta))&\leq&W_i(x(t_k-N^{3N}\delta))+\int_{I_k}d\tau\,D_{\dot x(\tau)}W_i(x(\tau))\leq W_i(x(t_k-N^{3N}\delta))-\delta |\mathcal{A}_i|.
\end{eqnarray}

This implies $\lim_{t_k\to +\infty} W_i(x(t_k+N^{3N}\delta))= -\infty$. This contradicts the fact that $W_i$ is bounded from below. We have now proved \eqref{e-AA}.

We now prove that \eqref{e-AA} ensures $W(x^*)=0$. For each $i\in V$, recall the definition of corresponding indexes  $j_i,l_i,m_i$ given above. For $i=1$, condition \eqref{e-AA} implies one of the following cases:
\begin{itemize}
    \item {\bf Case 1A)} the index $j_1$ satisfies $\|x_1^*-x_{j_1}^*\|=0$. This in turn implies that the only $j\in N_1^t(x^*)$ satisfies $\|x_1^*-x_j^*\|\leq \|x_1^*-x_{j_1}^*\|=0$. This in turn implies $a(\|x_1^*-x_j^*\|)=0$, i.e.\  $W_1(x^*)=0$.
    \item {\bf Case 1B)} the index $j_1$ satisfies $\|x_i^*-x_{j_1}^*\|\neq 0$.  Observe that, by construction of $j_1,l_1$, it holds $\|x_i(t)-x_{j_1}(t)\|=\|x_i(t)-x_{l_1}(t)\|$, thus by continuity it holds $\|x_i^*-x_{l_1}^*\|=\|x_i^*-x_{j_1}^*\|\neq 0$. Moreover, the definition of $j$ in \eqref{e-Dbase} implies that the following estimate holds 
    \begin{eqnarray}\label{e-split}
    &&a(\|x_1(t_k)-x_{j_1}(t_k)\|)[a(\|x_1(t_k)-x_{l_1}(t_k)\|)(x_{l_1}(t_k)-x_1(t_k))\\
    &&~~~~~~~~~~ -a(\|x_{j_1}(t_k)-x_{m_1}(t_k)\|)(x_{m_1}(t_k)-x_{j_1}(t_k))]\cdot(x_1(t_k)-x_{j_1}(t_k))\leq \nonumber\\
    &&    a(\|x_1(t_k)-x_{l_1}(t_k)\|)[a(\|x_1(t_k)-x_{l_1}(t_k)\|)(x_{l_1}(t_k)-x_1(t_k))\nonumber\\
    &&~~~~~~~~~~  -a(\|x_{l_1}(t_k)-x_{l_{l_1}}(t_k)\|)(x_{l_{l_1}}(t_k)-x_{l_1}(t_k))]\cdot(x_1(t_k)-x_{l_1}(t_k))\leq 0,\nonumber
    \end{eqnarray}
    where $l_{l_1}$ is the unique index in $N_{l_1}^t(x(t_k))$.
    The last inequality can be proved as in \eqref{e-DWi}. The left hand side of \eqref{e-split} converges to \eqref{e-AA} as $t_k\to+\infty$, thus it converges to zero. Then, the middle term converges to zero too, i.e.
    \begin{eqnarray}\label{e-allineato}
    -a(\|x_1^*-x_{l_1}^*\|)\|x_{l_1}^*-x_1^*\|^2=a(\|x_{l_1}^*-x_{l_{l_1}}^*\|)(x_{l_{l_1}}^*-x_{l_1}^*)\cdot (x_1^*-x_{l_1}^*).
    \end{eqnarray}
    Here we used the fact that $a(\|x_1^*-x_{l_1}^*\|)\neq 0$. 
    Since $\|x_1^*-x_{l_1}^*\|\geq \|x_{l_1}^*-x_{l_{l_1}}^*\|$ by construction of $l_{l_1}$ and $a$ is non-decreasing, the only possibility for \eqref{e-allineato} to hold is to have $\|x_{l_1}^*-x_{l_{l_1}}^*\|=\|x_1^*-x_{l_1}^*\|\neq 0$ and $(x_{l_{l_1}}^*-x_{l_1}^*)\cdot (x_1^*-x_{l_1}^*)=-\|x_{l_1}^*-x_1^*\|^2$, i.e.\ $x_1,x_{l_1},x_{l_{l_1}}$ being on the same line with $x_{l_1}$ as middle point. This also implies that $1,l_1,l_{l_1}$ are all different indexes. We relabel ${l_1},l_{l_1}$ as indexes $2,l_2$, for simplicity of notation.
\end{itemize}
We then apply the same idea to index $2$ (either coming from relabeling or not), and we have the following cases:
\begin{itemize}
    \item {\bf Case 2A)} The index $j_2$ satisfies $\|x_2^*-x_{j_2}^*\|=0$. Since $j_2\in A_2(x(t))$ for all $t\in I_k$ and $l_2$ is the unique element of $N_2^t(x(t))$, this implies $$\|x_2(t)-x_{l_2}(t)\|=\|x_2(t)-x_{j_2}(t)\|,$$
    thus $\|x_2^*-x_{l_2}^*\|=0$ by continuity. This implies that {\bf Case 1B)} is not compatible with {\bf Case 2A)}: indeed,  \eqref{e-allineato} implies $a(\|x_1^*-x_{l_1}^*\|)\|x_{l_1}^*-x_1^*\|^2=0$, thus $\|x_1^*-x_{l_1}^*\|=0$ and, by continuity, it holds $\|x_1^*-x_{j_1}^*\|=0$. 
    \item {\bf Case 2B)} The index $j_2$ satisfies $\|x_2^*-x_{j_2}^*\|\neq 0$. By following the reasoning of {\bf Case 1B)}, we find that $2,l_2,l_{l_2}$ are aligned, with $l_2$ being the middle point. This implies that, if both {\bf Case 1B)} and {\bf Case 2B)} hold, then $1,2,l_2,l_{l_2}$ are aligned, each with the same distance with respect to the previous one. Like in {\bf Case 1B)}, we also have that the four indexes are all distinct. This also allows to relabel $l_2,l_{l_2}$ as $3,l_3$, for simplicity of notation.
\end{itemize}

We now apply the same reasoning to all indexes $i=3,\ldots, N$ either after relabeling (due to {\bf Case $i$B}) or not. By incompatibility between {\bf Case $i$B} and {\bf Case $(i+1)$A}, we have the following structure: we first have $i$ cases, that are  {\bf Cases 1A-2A-...-$i$A}, then $N-i$ cases, that are {\bf Cases $(i+1)$B-...-$N$B}. We prove that $i=N$, by contradiction. Observe that {\bf Cases $(i+1)$B-...-$N$B} force us to have agents $i+1,\ldots,N,l_N$ aligned on the same line, each with the same distance with respect to the previous one. Since the number of agents is $N$, the agent $l_N$ is one among $1,\ldots,N$. By the alignment condition, it cannot be any of the agents $i+1,\ldots,N$, hence {\bf Case $l_N$A} holds. By incompatibility of conditions described above, {\bf Case $N$B} cannot hold. This raises a contradiction. As a consequence, for each $i\in V$ the {\bf Case $i$A} is satisfied. This means that for each $i\in V$ there exists $j\neq i$ such that $x_i^*=x_j^*$. This also implies $W_i(x^*)=0$. In particular, $x^*$ satisfies the second statement of this proposition.

We are now left to prove that the $\omega$-limit is reduced to a single point, i.e.\ that $x^*$ given above is $x^\infty$ in the first statement of this proposition. Define the following equivalence relation: $i\sim j$ if $x_i^*=x_j^*$. Observe that each class of equivalence $[i]_\sim$ is composed of at least two elements.
We have two possibilities:
\begin{itemize}
    \item There exists a single class of equivalence $[i]_\sim$. Then, for each $\eps>0$ there exists $t_k$ such that $\|x_i(t_k)-x_i^*\|=\|x_i(t_k)-x_1^*\|<\eps$. Since this holds for all indexes, then the support of the solution $x(t_k)$ is contained in $B(x_1^*,\eps)$. Since the support is non-increasing, due to Proposition~\ref{p-contractive}, then the solution $x(t)$ is contained in $B(x_1^*,\eps)$ for all $t\geq t_k$. Since this condition holds for all $\eps>0$, then $x_i(t)\to x_1^*=x_i^*$.
    \item There exist at least two classes of equivalence $[i]_\sim\neq [j]_\sim$. Define the minimal distance between clusters as $5\lambda:=\min_{i\not\sim j}\|x_i^*-x_j^*\|$, that satisfies $\lambda>0$. By convergence of $x_i(t_k)$ to $x_i^*$, there exists $\bar k$ sufficiently large to have $\|x_i(t_{\bar k})-x_i^*\|<\lambda$ for all $i\in V$. As a consequence, the following {\bf cluster separation condition} holds: 
    \begin{center}
        
    If $i\sim j$, then it holds  $\|x_i(t_{\bar k})-x_j(t_{\bar k})\|<2\lambda$.
    If $i\not \sim j$, then it holds $\|x_i(t_{\bar k})-x_j(t_{\bar k})\|>3\lambda$.
        \end{center}

    It is now easy to prove that, for all $t\geq t_{\bar k}$, the same cluster separation condition holds too, since interactions between agents of different clusters do not occur: the proof is similar to Proposition~\ref{p-contractive} and is omitted. As a consequence, each of the cluster acts as an independent system starting from $t_{\bar k}$. In particular, we can apply Proposition~\ref{p-contractive} to each cluster independently: similarly to the previous case, for each $\eps>0$ there exists $k\geq {\bar k}$ such that for each class of equivalence $[i]_\sim$ the support of $\{x_j(t)\mbox{~~s.t.~~} j\in [i]_\sim\}$ is contained in $B(x_i^*,\eps)$ for all $t\geq t_k$. By letting $\eps \to 0$, we have $x_j(t)\to x_i^*$.
\end{itemize}
In both cases, we have proved that the $\omega$-limit of $x(t)$ is reduced to $x^*$. Thus, by choosing $x^\infty=x^*$ we have that the statement is proved.

\end{proof}

\section{Future directions}
In this paper we explored various concepts of solutions for discontinuous differential equations, motivated by social dynamics models.
In particular, we focused on the so-called bounded confidence models, where each agent is interacting either with neighbors within a fixed distance (metric case) or with the $\kappa$ closest ones
(topological case).
As per the concepts of solutions we focused on
Caratheodory and Krasovsky, after proving that the set of Filippov solutions coincides with that of Krasovsky solutions for the considered models.

Existence of solutions and uniqueness for almost every initial datum are proved in Krasovsky and
Caratheodory sense for both models.
We also explored properties of solutions
such as preservation of the average,
contractivity of support and convergence to cluster points. 
Contractivity of the support always holds true,
the other properties hold for the metric case
(and both concepts of solutions),
while they fail for the topological case
with the exception of convergence to cluster points that holds for Caratheodory solutions if $\kappa=1$.

Future investigations may include:
\begin{itemize}
    \item Exploring existence, uniqueness and properties of trajectories for other concepts of solutions, such as limit of Euler or CLSS, stratified solutions and others \cite{piccoli_rossi_2020};
    \item Studying the implications of our results to approximated solutions produced by numerical schemes;
    \item Considering the topological-metric case, where each agent interacts with the closest $\kappa$ neighbors if they are within a fixed distance;
    \item Extending the scope of our analysis to include dynamical models with other types of discontinuities, as those generated by quantization~\cite{ceragioli2018discontinuities} or hybrid setting~\cite{frasca2019hybrid}.
\end{itemize}

\bibliographystyle{abbrv}

\end{document}